\documentclass[12pt,reqno]{amsart}
\usepackage{amssymb,amsmath,amsthm,amsfonts,txfonts}
\usepackage[english]{babel}
\usepackage{mathrsfs,a4wide,dsfont}
\usepackage[utf8]{inputenc}
\usepackage{color}
\usepackage{mathtools}
\usepackage{comment}
\usepackage[english]{babel}
\usepackage{verbatim}

\theoremstyle{plain}
\newtheorem{theorem}{Theorem}[section]
\newtheorem{lemma}[theorem]{Lemma}
\newtheorem{proposition}[theorem]{Proposition}
\newtheorem{example}[theorem]{Example}
\newtheorem{corollary}[theorem]{Corollary}
\newtheorem{remark}[theorem]{Remark}
\newtheorem{definition}[theorem]{Definition}
\theoremstyle{definition}
\theoremstyle{remark}
\numberwithin{equation}{section}
\newenvironment{oss}{\begin{remark} \begin{rm}}{\end{rm} \end{remark}}

\newcommand{\lip}{\mathrm{Lip}}
\newcommand{\BV}{\mathrm{BV}}
\renewcommand{\phi}{\varphi}
\newcommand{\R}{{\mathbb{R}}}

\newcommand{\bbR}{{\mathbb{R}}}
\newcommand{\bbN}{{\mathbb{N}}}
\def \R {{\mathbb {R}}}

\newcommand{\bbH}{\mathbb{H}}

\newcommand{\bbS}{\mathbb{S}}

\newcommand{\average}{{\mathchoice {\kern1ex\vcenter{\hrule height.4pt
width 6pt
depth0pt} \kern-9.7pt} {\kern1ex\vcenter{\hrule height.4pt width 4.3pt
depth0pt}
\kern-7pt} {} {} }}

\newcommand{\HH}{\mathcal{H}}
\newcommand{\LL}{\mathcal{L}}

\newcommand{\ud}{\,\mathrm{d}}

\newcommand{\res}{\mathop{\hbox{\vrule height 7pt width .5pt depth 0pt
			\vrule height .5pt width 6pt depth 0pt}}\nolimits}

\newcommand{\tr}[1]{\langle{#1}\rangle}

\newcommand{\X}{{\bf X}}
\newcommand{\xvett}{\X}
\newcommand{\xva}{\xvett^\ast}

\DeclareMathOperator*{\argmin}{arg\,min}

\newcommand{\dive}{\mathrm{div}\,}

\newcommand{\atx}[1]{{#1}_{\tau,\xi}^\ast}

\newcommand{\ut}{\overline u_\tau^*}

\newcommand{\norm}[1]{\|{#1}\|}
\newcommand{\area}{\mathcal G}
\newcommand{\cc}{\mathscr M}

\newcommand{\G}{\mathscr G}

\newcommand{\bd}{\begin{definition}}
\newcommand{\ed}{\end{definition}}
\newcommand{\bt}{\begin{theorem}}
\newcommand{\et}{\end{theorem}}
\newcommand{\bp}{\begin{proposition}}
\newcommand{\ep}{\end{proposition}}
\newcommand{\br}{\begin{remark}}
\newcommand{\er}{\end{remark}}
\newcommand{\bl}{\begin{lemma}}
\newcommand{\el}{\end{lemma}}
\newcommand{\bc}{\begin{corollary}}
\newcommand{\ec}{\end{corollary}}
\newcommand{\beq}{\begin{equation}}
\newcommand{\eeq}{\end{equation}}

\usepackage{xcolor}

    \let\TeXchi\chi
\newbox\chibox
\setbox0 \hbox{\mathsurround0pt $\TeXchi$}
\setbox\chibox \hbox{\raise\dp0 \box 0 }
\def\chi{\copy\chibox}

\title[Lipschitz minimizers of functionals]{Lipschitz minimizers for a class of integral functionals under the bounded slope condition}
\author[S.\, Don]
{Sebastiano Don}
\address[Sebastiano Don]{Mathematisches Institut, Sidlerstrasse 12, 3012 Bern, Switzerland}
\email[]{sebastiano.don@math.unibe.ch}
\author[L.\,Lussardi]
{Luca Lussardi}
\address[Luca Lussardi]{Dipartimento di Scienze Matematiche ``G.L.\,Lagrange'', Politecnico di Torino, corso Duca degli Abruzzi 24, 10129 Torino, Italy}
\email[]{luca.lussardi@polito.it}
\author[A.~Pinamonti]
{Andrea Pinamonti}
\address[Andrea Pinamonti]{Dipartimento di Matematica, Universit\`a degli Studi di Trento, via Sommarive 14, 38123 Povo (Trento), Italy}
\email[]{andrea.pinamonti@unitn.it}
\author[G.~Treu]
{Giulia Treu}
\address[Giulia Treu]{Dipartimento di Matematica ``Tullio Levi Civita'', Universit\`a degli Studi di Padova, via Trieste 63, 35121 Padova, Italy}
\email[]{giulia.treu@unipd.it}



\thanks{
 S.D. was partially supported by the Academy of Finland (grant
	288501
	`\emph{Geometry of subRiemannian groups}' and by grant
	322898
	`\emph{Sub-Riemannian Geometry via Metric-geometry and Lie-group Theory}') and by the European Research Council
	(ERC Starting Grant 713998 GeoMeG `\emph{Geometry of Metric Groups}'). S.D. was also partially supported by the Swiss National Foundation grant 200020\_191978. A.P. is partially supported by the INdAM-GNAMPA 2020 project Convergenze variazionali per funzionali e operatori dipendenti da campi vettoriali.
}

\begin{document}

\baselineskip3.4ex

\vspace{0.5cm}
\begin{abstract}
We consider the functional $\int_\Omega g(\nabla u+\X^\ast)\,\ud\LL^{2n}$ where $g$ is convex and $\X^\ast(x,y)=2(-y,x)$ and we study the minimizers in $BV(\Omega)$ of the associated Dirichlet problem. We prove that, under the bounded slope condition on the boundary datum, and suitable conditions on $g$, there exists a unique minimizer which is also Lipschitz continuous.
\end{abstract}

\maketitle
\section{Introduction}
In the present paper we are interested in the study of the Lipschitz regularity of minimizers of a class of functionals starting from the regularity of the boundary datum without assuming neither ellipticity nor the growth conditions on the lagrangian: the literature on this subject is extremely rich, we address the interested reader to \cite{BousquetACV,BouCla,CellinaBSC,ClarkeSNS,FiaschiTreu,LM, MTpams, MTjde, MTccm} and references therein for an overview.
Our analysis moves from a recent paper by Pinamonti et al.\,\cite{PSTV} where the area functional for the $t$-graph of a function $u\in W^{1,1}(\Omega)$ in the sub-Riemannian Heisenberg group $\mathbb{H}^n=\R^n_x \times \R^n_y \times \R_t$ is investigated (see also further references in \cite{PSTV} on the Heisenberg's literature). Precisely, if $\Omega \subset \mathbb{R}^{2n}$ is open with Lipschitz boundary and ${\bf X}^*(x,y)\coloneqq2(-y,x)\in \mathbb{R}^{2n}$ they consider the functional $\mathscr A\colon W^{1,1}(\Omega)\to \mathbb{R}$ defined by
\[
\mathscr A(u)=\int_\Omega |\nabla u+{\bf X}^*|\,\ud\LL^{2n}.
\]
It was shown in \cite{SCV} that because of the linear growth in the gradient variable, the natural variational setting for the functional $\mathscr A$ is $BV(\Omega)$, the space of functions of bounded variation in $\Omega$. More precisely, it has been proved that the $L^1-$relaxation of $\mathscr{A}$ is 
\[
\mathscr A(u)=\int_\Omega |\nabla u+{\bf X}^*|\,\ud\LL^{2n}+|D^su|(\Omega), \quad u\in BV(\Omega)
\]
where $|D^s u|$ denotes the total variation of the singular part of the distributional derivative of $u$.
In \cite{PSTV}, the authors investigate a suitable Dirichlet problem for $\mathscr A$. Precisely, they show that the problem
\[
\min\left\{\mathscr A(u) : u\in BV(\Omega),\,u_{|\partial \Omega}=\varphi\right\}
\]
has a unique solution which is also Lipschitz continuous if $\varphi\in L^1(\partial\Omega)$ satisfies the so-called {\it bounded slope condition} (see Section $4$ below for the definition).

In the present paper we are interested in the more general case of functionals of type
\[
\mathscr G(u)=\int_\Omega g(\nabla u+{\bf X}^*)\,\ud\LL^{2n}
\]
where $g:\mathbb{R}^{2n}\to\mathbb{R}$ is convex but not necessarily strictly convex. In particular, we want to study the Dirichlet problem associated to $\mathscr G$ assuming the bounded slope condition on the boundary datum., i.e.
\begin{equation}\label{problem}
\min\left\{\mathscr G(u) : u\in BV(\Omega),\,u_{|\partial \Omega}=\varphi\right\}
\end{equation}
where $\varphi\in L^1(\partial\Omega)$.
Notice that the existence of minimizers is in general not guaranteed even for smooth boundary data on smooth domains and $g(z)=|z|$, see \cite[Example 3.6]{SCV}. The problem of the uniqueness of minimizers has been considered in \cite{Beck}, where the authors study integral functionals with linear growth defined on vector-valued functions with uniformly elliptic lagrangian and they prove that every minimizer is unique up to additive constants. 
We divide the paper in two parts: in the first, which is also the longest one, we will treat the case in which $g$ has linear growth whereas in the second we will describe how to modify the proof in order to deal with the case in $g$ has superlinear growth. 

In the first case the proof follows the line of the one given in \cite{PSTV}. We first observe that the relaxed functional of $\mathscr{G}$ in the $L^1$-topology can be written as
\[
\mathcal G (u)=\int_\Omega g(\nabla u+\X^\ast)\,\ud\LL^{2n}+\int_\Omega g^\infty\left(\frac{dD^su}{d|D^su|}\right)d|D^su|+\int_{\partial\Omega}g^\infty((\varphi-u_{|\partial\Omega})\nu_\Omega)\,d\mathcal H^{2n-1}
\]
where $g^{\infty}:\mathbb{R}^{2n}\to [0,+\infty)$ denotes the recession function of $g$ and $\nu_{\Omega}$ is the unit outer normal to $\partial\Omega$. 
Then we prove that $\mathcal G$ admits a minimizer in $BV(\Omega)$. In order to do so we need to impose some restrictions on $g$, in particular we assume that 
\[
g\left(\frac{\xi_1+\xi_2}{2}\right)=\frac{g(\xi_1)+g(\xi_2)}{2} \Longrightarrow \xi_1=\lambda\xi_2
\]
\[
\text{[$p\in \partial g(\xi_2)$ and $g^\infty(\xi_1)=\langle p,\xi_1\rangle] \Longrightarrow \xi_1=\lambda\xi_2$}.
\]
Notice that these two assumptions say that, in some sense, $g$ is non too far from a strictly convex function and this is the reason why we expect to get more regularity of the minimizers. Finally, we prove that under the bounded slope condition the minimizer is actually unique and Lipschitz regular.
The approach used to get Lipschitz regularity of minimizers of $\mathcal G$ is inspired by some classical and well known results in the Calculus of Variations (see \cite{Hilbert} and \cite{Haar} and also \cite{HartmanBSC,HartmanConvex,HartStamp,Stampacchia}). The main idea (see \cite[Chapter 1]{G}) is that the bounded slope condition assumed on the boundary data produces a control of the Lipschitz constant of the minimizer.  

 We point out that crucial points in our analysis are the validity of comparison principles between minimizers, the invariance of minimizers under translations of the domain, and the fact that the affine functions are the unique minimizers among all $BV$-functions with their boundary data, which has been one of the main difficulty in our investigation (see Proposition \ref{propmin2}). 
 
If $g$ growths more than linearly we prove that the functional $\mathscr G$ has a unique minimizer in $\varphi +W_0^{1,1}(\Omega)$ which is also Lipschitz continuous if $\varphi$ satisfies the bounded slope condition. In this case the proof is more standard and we only sketch it.\\

The paper is organized as follows. In Section \ref{recallBV} we recall some basic facts about functions of bounded variation and the notion of trace. In Section \ref{sec:linear} we introduce the functional in the linear growth case, its relaxation and we investigate the main submodularity-type property of $\mathcal G$. Section \ref{ultSec} is devoted to the proof of our main Theorem (Theorem\ \ref{mainteo-ex6.7} below) in the linear growth case. Finally, Section \ref{lastSec} contains some results about the superlinear growth case.

\section{Preliminaries}\label{recallBV}

\subsection{Functions of Bounded Variation and traces}

The aim of this section is to recall some basic properties of the space of functions of bounded variation; we refer to the monographs \cite{AFP,Giusti} for a more extensive account on the subject as well as for proofs of the results we are going to recall.

Let $\Omega$  be an open set in $\R^n$. We say that $u\in L^1(\Omega)$ has {\em bounded variation} in $\Omega$ if
\begin{equation}\label{ilsup}
\sup\left\{\int_{\Omega} u\ \dive \varphi\ dx\ |\ \varphi\in C^1_c(\Omega,\mathbb{R}^n), \norm{\varphi}_\infty\leq 1\right\}<+\infty;
\end{equation}
equivalently, $u$ has bounded variation if there exist a $\R^n$-valued Radon measure $Du\coloneqq(Du_1,\dots, Du_n)$ in $\Omega$ which represents the distributional derivatives of $u$, i.e.,
\[
\int_{\Omega} u\frac{\partial\varphi}{\partial x_i}\ud\LL^n=-\int_{\Omega} \varphi\  \ud D_i u\quad \forall \varphi\in C^1_c(\Omega),\ \forall i=1,\ldots, n.
\]
The space of functions with bounded variation in $\Omega$ is denoted by $\BV(\Omega)$. By definition, $W^{1,1}(\Omega)\subset \BV(\Omega)$ and $Du=\nabla u\,\LL^n$ for any $u\in W^{1,1}(\Omega)$.

We denote by $|Du|$ the total variation of the measure $Du$; $|Du|$ defines a finite measure on $\Omega$ and the supremum in \eqref{ilsup} coincides with $|Du|(\Omega)$.

It is well-known that $\BV(\Omega)$ is a Banach space when endowed with the norm
\begin{align}\label{normBV}
    \norm{u}_{\BV(\Omega)}\coloneqq\norm{u}_{L^1(\Omega)}+ |Du|(\Omega).
\end{align}

We say that $u\in L^1_{loc}(\Omega)$ has an {\em approximate limit} $z\in \bbR$ at $x\in\Omega$ if
\begin{equation}\label{defapplim}
\lim_{\rho\to 0^+}\fint_{B(x,\rho)} |u-z| \ud\LL^n=0.
\end{equation}
The set $S_u$ of points where $u$ has no approximate limit is called {\em approximate discontinuity set} of $u$; for any $x\in \Omega\setminus S_u$, we denote by $\tilde u(x)$ the unique $z$ for which \eqref{defapplim} holds. By the Lebesgue Theorem we have $\LL^n(S_u)=0$.


\noindent Moreover, we say that $u$ has an {\em approximate jump point} at $x\in\Omega$ if there exist $\nu\in \bbS^{n-1}$ and $a,b\in \bbR,\ a\neq b$  such that
\[
    \lim_{\rho\to 0^+}\fint_{B(x,\rho;\nu)^+}|u-a| \ud\LL^n=0,\quad \lim_{\rho\to 0^+}\fint_{B(x,\rho;\nu)^-}|u-b| \ud\LL^n=0
\]
where
\[
\begin{split}
    &B(x,\rho;\nu)^+\coloneqq\{y\in B(x,\rho)\ |\ \tr{ y-x,\nu}>0\}\\
    &B(x,\rho;\nu)^-\coloneqq\{y\in B(x,\rho)\ |\ \left\langle y-x,\nu\right\rangle<0\}.
\end{split}
\]
We observe that the triple $(a,b,\nu)$ is uniquely determined up to a permutation of $(a,b)$ and a change of sign of $\nu$; we denote it by $(u^+(x),u^-(x), \nu_u(x))$. The set of approximate jump points of $u$ is denoted by $J_u$; clearly, $J_u\subset S_u$.

\begin{remark}{
Depending on the context, we will sometimes use the symbols $u^+,u^-$ also to denote the positive part $u^+\coloneqq\max\{0,u\}$ and the negative part $u^-\coloneqq\max\{0,-u\}$ of a real function $u$. This will not generate confusion.
}\end{remark}

When $u$ has bounded variation in $\Omega$, the set of approximate jump points $J_u$ enjoys much finer regularity properties. First, there holds
\begin{equation}\label{negligSuJu}
|Du|(S_u\setminus J_u)=\HH^{n-1}(S_u\setminus J_u)=0\,,
\end{equation}
where $\HH^{n-1}$ denotes the $(n-1)$-dimensional {\em Hausdorff measure} on $\R^n$ (see e.g. \cite{AFP} or \cite{Giusti}).
Moreover, by the Federer-Vol'pert Thoerem, see \cite[Theorem 3.78]{AFP}, $J_u$ (and, consequently, $S_u$) is $(n-1)$-rectifiable, i.e., $\HH^{n-1}(J_u)<\infty$ and there exist $N\subset\R^n$ and a countable family of hypersurfaces $\{S_j:j\in\bbN\}$ of class $C^1$ such that
\[
J_u\subset N\cup\bigcup_{j=0}^\infty S_j\quad\text{and}\quad \HH^{n-1}(N)=0\,.
\]
 It turns out that $\nu_u$ corresponds ($\HH^{n-1}$-a.e.\ and up to a sign) to a unit normal to $J_u$, i.e., for $\HH^{n-1}$-a.e. $x\in J_u$, there holds
\[
\nu_u(x)=\pm\nu_{S_i}(x) \text{ if }x\in S_i\setminus \bigcup_{j=0}^{i-1} S_j,\quad\forall i\in\bbN\,.
\]

By the Radon-Nikodym Theorem, if $u\in \BV(\Omega)$ one can write $Du=D^au+D^su$, where $D^au$ is the absolutely continuous part of $Du$ with respect to $\LL^n$ and $D^su$ is the singular part of $Du$ with respect to $\LL^n$. We denote by $\nabla u\in L^1(\Omega)$ the density of $D^au$ with respect to $\LL^n$, so that $D^au=\nabla u\,\LL^n$. We are now in a position to state the following result:

\begin{theorem}\label{tuttoBV}
Let $u\in \BV(\Omega)$; then $u$ is {\em approximately differentiable} at a.e. $x\in\Omega$ with approximate differential $\nabla u(x)$, i.e.,
\[
    \lim_{\rho\to 0^+}\fint_{B(x,\rho)} \frac{|u(y)-\tilde u(x)-\left\langle \nabla u(x), y-x\right\rangle|}{\rho}\ \ud\LL^n=0\qquad\text{for $\LL^n$-a.e. }x\in\Omega\,.
\]
Moreover, the  decomposition $D^su=D^ju+D^cu$ holds, where
\[
D^ju\coloneqq D^su\res J_u=(u^+-u^-)\nu_u\HH^{n-1}\res J_u,\qquad D^cu\coloneqq D^su\res (\Omega\setminus S_u)
\]
are called respectively the {\em jump part} and the {\em Cantor part} of the derivative $Du$.
\end{theorem}

Notice that $D^au,D^cu,D^ju$ are mutually singular; in particular
\[
|D^au|=|\nabla u|\LL^n,\quad|D^ju|=|u^+-u^-|\HH^{n-1}\res J_u
\]
and
\[
|Du|=|D^au|+|D^cu|+|D^ju|
\]
because the total variation of a sum of mutually singular measures is the sum of their total variations.

%

\medskip
In what follows we recall a few basic facts about boundary trace properties of BV functions; we refer again to \cite{AFP} and \cite{Giusti} for more details.

Let $\Omega\subset\R^n$ be a fixed open set with bounded Lipschitz regular boundary; the spaces $L^p(\partial\Omega), p\in[1,+\infty]$, will be always understood with respect to the (finite) measure $\HH^{n-1}\res\partial\Omega$. It is well-known that for any $u\in \BV(\Omega)$ there exists a (unique) function $u_{|\partial\Omega}\in L^1(\partial\Omega)$ such that, for $\HH^{n-1}$-a.e. $x\in \partial\Omega$,
\[
\lim_{\rho\to0^+}\rho^{-n} \int_{\Omega\cap B(x,\rho)}|u-u_{|\partial\Omega}(x)|\ud\LL^n=\lim_{\rho\to0^+}\fint_{\Omega\cap B(x,\rho)}|u-u_{|\partial\Omega}(x)|\ud\LL^n=0\,.
\]
The function $u_{|\partial\Omega}$
 is called {\em trace} of $u$ on $\partial\Omega$. The trace operator $u\mapsto u_{|\partial\Omega}$ is linear and continuous between $(\BV(\Omega),\norm\cdot_{\BV})$ and $L^1(\partial\Omega)$; actually, it is continuous also when $\BV(\Omega)$ is endowed with the (weaker) topology induced by the so-called {\em strict convergence}, see \cite[Definition 3.14]{AFP}.

\begin{remark}\label{traccesupinf}{\rm
It is well-known that, if $u_1,u_2\in \BV(\Omega)$, then $\overline u\coloneqq\max\{u_1,u_2\}$ and $\underline u\coloneqq\min\{u_1,u_2\}$ belong to $\BV(\Omega)$; moreover, one can show that
\[
\overline u_{|\partial\Omega}=\max\{u_{1|\partial\Omega},u_{2|\partial\Omega}\},\qquad \underline u_{|\partial\Omega}=\min\{u_{1|\partial\Omega},u_{2|\partial\Omega}\}\,.
\]
The proof of this fact follows in a standard way from the very definition of traces.
}\end{remark}

Since $Du\ll|Du|$ we can write $Du=\sigma_u|Du|$ for a $|Du|$-measurable function
\[
\sigma_u\colon\Omega\to\bbS^{n-1}.\medskip
\]
With this notation one also has
\begin{equation}\label{defvartr}
\int_\Omega u\,\dive\varphi\ud\LL^n = -\int_\Omega \langle \sigma_u,\varphi\rangle \ud |Du| + \int_{\partial\Omega} u_{|\partial\Omega}\, \langle\varphi,\nu_\Omega\rangle \ud\HH^{n-1},\qquad\forall \varphi\in C^1_c(\R^n;\R^n)
\end{equation}
where $\nu_\Omega$ is the unit outer normal to $\partial\Omega$.

Finally, we recall the following fact, whose proof essentially follows from \eqref{defvartr}.

\begin{proposition}[{\cite[Remark 2.13]{Giusti}}]\label{tracceatratti}
Assume that $\Omega$ and $\Omega_0$ are open subsets of $\R^n$ with bounded Lipschitz boundary and such that $\Omega\Subset\Omega_0$. If $u\in \BV(\Omega)$ and $v\in \BV(\Omega_0\setminus\overline\Omega)$, then the function
\[
f(x):=
\left\{\begin{array}{ll}
u(x) & \text{if }x\in\Omega\\
v(x) & \text{if }x\in\Omega_0\setminus\overline\Omega
\end{array}\right.
\]
belongs to $\BV(\Omega_0)$ and
\[
|Df|(\partial\Omega)= |D^jf|(\partial\Omega) = \int_{\partial\Omega}|u_{|\partial\Omega}-v_{|\partial\Omega}|\ud\HH^{n-1}\,,
\]
where we have used the notation $v_{|\partial\Omega}$ to mean $(v_{|\partial(\Omega_0\setminus\overline\Omega)})\res{\partial\Omega}$.
\end{proposition}
For any $z=(x,y)\in\R^{2n}$, we define $z^\ast \coloneqq (-y,x)$. Let $\X^\ast\colon \mathbb R^{2n} \to \mathbb R^{2n}$ be given by $\X^\ast(z)\coloneqq2z^\ast$. We conclude this section with the next lemma which can be extracted from the proof of \cite[Thm.\,5.5]{PSTV}.
\begin{lemma}\label{0hom}
Let $R>0$ and $u\in BV(B_R(0))$ with $u=0$ on $\partial B_R(0)$. Assume that there exists a $|Du|$-measurable function $\lambda \colon B_R(0) \to \R$ such that
\[
\frac{dDu}{d|Du|}=\lambda\, \X^* \quad \textrm{$|Du|$-a.e.\,on $B_R(0)$}.
\]
Then $u=0$.
\end{lemma}
\section{The linear growth case}\label{sec:linear}

Throughout this section we assume that $g\colon\mathbb{R}^{2n}\to \mathbb{R}$ is a positive convex function with linear growth, namely
\begin{equation}\label{crescita}
\frac1C|z|\le g(z) \le C(1+|z|),
\end{equation}
for a constant $C\geq1$ and for any $z\in \mathbb{R}^{2n}$.
Moreover, defining the {\it recession function} of $g$ as the function $g^\infty\colon \R^{2n}\to [0,+\infty)$ given by
\begin{equation*}
g^\infty(p)\coloneqq\lim_{t\to+\infty}\frac{g(tp)}{t}.
\end{equation*}
Note that, since $g(0)<\infty$, our definition of $g^{\infty}$ coincides with the one given in \cite[Definition 2.32]{AFP}. As proved in \cite{AFP}, the recession function is positively homogeneous of degree $1$, convex and lower semicontinuous. In particular, $g^{\infty}$ satisfies the following inequalities
\begin{align}\label{sublin}
&g^{\infty}(p)\leq g^{\infty} (q)+g^{\infty}(p-q),\qquad \forall p,q\in \mathbb{R}^{2n},\\
\label{crescitainfty}
&\frac{1}{C}|p|\leq g^{\infty}(p)\leq C|p|,\qquad \forall p\in \mathbb{R}^{2n}.
\end{align}
Since by \cite[Proposition 2.32]{Da}, $g$ is Lipschitz continuous then denoting by $L_g$ its Lipschitz constant we get
	\[
	\left|g(tp)-g(tp+z)\right|\le L_g|z|
	\]
which implies that for any $z,p\in\R^{2n}$ we have
	\begin{equation}\label{remginf}
	g^\infty(p)=\lim_{t\to+\infty}\frac{g(tp+z)}{t}.
	\end{equation}
We consider the following conditions:
\begin{itemize}
	\item[\rm(A)] If $\xi_1,\xi_2\in \R^{2n}$ are such that
	\begin{equation} \label{sp}
	g\left(\frac{\xi_1+\xi_2}{2}\right)=\frac{g(\xi_1)+g(\xi_2)}{2},
	\end{equation}
	then there exists $\lambda\in \mathbb{R}$ such that $\xi_1=\lambda\xi_2$.
	\item[\rm(B)] If $\xi_1,\xi_2\in \R^{2n}$ and $p\in \partial g(\xi_2)$ are such that
	\begin{equation}\label{assB}
	g^\infty(\xi_1)=\langle p,\xi_1\rangle
	\end{equation}
	then there exists $\lambda\in \mathbb{R}$ such that $\xi_1=\lambda\xi_2$. Here $\partial g(q)$ denotes the subdifferential of $g$ at the point $q$.
\end{itemize}
\begin{example} Let $f\colon[0,+\infty)\to \mathbb{R}$ be a strictly convex and increasing function such that there exists  $C>1$ satisfying 
\[
\frac{1}{C}s\leq f(s)\leq C(s+1)
\]
for any $s\in [0,+\infty)$. Consider the function $g\colon\mathbb{R}^{2m}\to \mathbb{R}$ defined by $g(z):=f(|z|)$. We claim that $g$ satisfies conditions $(A)$ and $(B)$. Indeed, for any $\xi_1, \xi_2\in\mathbb{R}^{2m}$ satisfying \eqref{sp} we get
\begin{equation}
f\left(\frac{|\xi_1|}{2}+\frac{|\xi_2|}{2}\right)\leq \frac{1}{2}\left( f(|\xi_1|)+f(|\xi_2|)\right)=f\left(\frac{|\xi_1+\xi_2|}{2}\right)\leq f\left(\frac{|\xi_1|}{2}+\frac{|\xi_2|}{2}\right)
\end{equation}
from which we infer $|\xi_1+\xi_2|=|\xi_1|+|\xi_2|$ and the thesis follows. To prove condition $(B)$, we start observing that by \cite[Example 16.73]{BCconvex} we have
\begin{equation}
\partial g(\xi)=\begin{cases}
                  \left\{ \frac{\alpha}{|\xi|}\xi\ |\ \alpha\in \partial f(|\xi|) \right\}, & \mbox{if } \xi\neq 0 \\
                  B(0,\rho), & \mbox{if } \xi=0
                \end{cases}
\end{equation}
where $\rho\in [0,+\infty)$ is such that $\partial f(0)=[-\rho,\rho]$. Moreover a direct computation gives
\begin{equation}\label{infrt}
g^{\infty}(\xi)=f^{\infty}(|\xi|).
\end{equation}
Let us now consider $\xi_1,\xi_2\in\mathbb{R}^{2n}$ and $p\in \partial g(\xi_2)$ such that $g^{\infty}(\xi_1)=\langle p,\xi_1\rangle$. If $\xi_2=0$ then the strict convexity of $g$ gives that for any $p\in \partial g(0)$ the function $h(\xi):=g^{\infty}(\xi) - \langle p, \xi\rangle$ 
is strictly convex and nonnegative and $h(0)=0=h(\xi_1)$ and the conclusion follows.
 If $\xi_2\neq 0$ then $p=\alpha \frac{\xi_2}{|\xi_2|}$ for some $\alpha\in \partial f(|\xi_2|)$ and $\alpha>0$. By \eqref{infrt} and the fact that $f^{\infty}$ is $1-$homogeneous we get
\[
|\xi_1||\xi_2| f^{\infty}(1)=\alpha \langle \xi_1,\xi_2\rangle \leq \alpha |\xi_1||\xi_2|
\]
thence $\alpha\geq f^{\infty}(1)$. On the other hand, the convexity of $f$ gives
\[
\frac{f(s)}{s}\geq \frac{f(|\xi_2|)}{s}+\frac{\alpha(s-|\xi_2|)}{s}
\]
for every $s\in (0,+\infty)$, and letting $s\to +\infty$ we get $f^{\infty}(1)\geq \alpha$ which coupled with the previous inequality yields $f^{\infty}(1)=\alpha$ and conclusion follows.
\end{example}
\begin{example}
Let $a,b \in (0,+\infty)$. We claim that the function $g\colon\mathbb{R}^2\to [0,+\infty)$ defined by
\begin{equation}
g(z_1,z_2)=\sqrt{\frac{z_1^2}{a^2}+\frac{z_2^2}{b^2}}
\end{equation}
satisfies \eqref{crescita}, conditions $(A)$ and $(B)$. Indeed, for any $z\in \mathbb{R}^2$
\begin{equation}
\min\left\{\frac{1}{a},\frac{1}{b}\right\}|z|\leq g(z)\leq \max\left\{\frac{1}{a},\frac{1}{b}\right\}|z|
\end{equation}
and $g$ is convex and it satisfies $(A)$ by a direct computation. In order to prove condition $(B)$ we start observing that, being $g$ $1$-homogeneous and in $C^{\infty}(\mathbb{R}^2\setminus \{(0,0)\})$, we have $g^{\infty}(z)=g(z)$ for any $z\in \mathbb{R}^2$ and
$\partial g(z)=\left\{\left(\frac{z_1}{a^2g(z)}, \frac{z_2}{b^2g(z)}\right)\right\}$ for any $z\in \mathbb{R}^2\setminus\{(0,0)\}$. Let $\xi=(\xi_1,\xi_2)\in\mathbb{R}^2$, $(\eta_1,\eta_2)\in \mathbb{R}^2\setminus\{(0,0)\}$ and $(p_1,p_2)=\left(\frac{\eta_1}{a^2g(\eta)}, \frac{\eta_2}{b^2g(\eta)}\right)$ be such that $g^{\infty}(\xi)=\langle p,\xi\rangle$, namely
\begin{equation}
\sqrt{\frac{\xi_1^2}{a^2}+\frac{\xi_2^2}{b^2}}=\frac{\eta_1\xi_1}{a^2g(\eta)}+ \frac{\eta_2\xi_2}{b^2g(\eta)}
\end{equation}
which immediately implies that $\xi_1\eta_2=\eta_1\xi_2$ and the thesis follows. On the other hand, let $\xi=(\xi_1,\xi_2)$ and $\eta=(\eta_1,\eta_2)$ be such that
\begin{equation}
\eta \in\partial g((0,0))\qquad \mbox{and}\qquad g(\xi)=\langle \eta, \xi\rangle.
\end{equation}
Since the function $f(z)=g(z)-\langle p,z\rangle$ is convex, $1$-homogeneous, nonnegative and $f(\xi)=f((0,0))=0$, then one has $\xi=(0,0)$.
\end{example}
Let $\Omega\subset \R^{2n}$ be bounded, open and with Lipschitz boundary.
We consider the functional $\G_{\Omega} \colon W^{1,1}(\Omega)\to [0,+\infty]$ defined by
\begin{equation}\label{sobolev}
\G_{\Omega}(u)\coloneqq\int_\Omega g(\nabla u+\X^*)\,\ud\LL^{2n}
\end{equation}
where we recall that $\X^\ast(z)=2(-y,x)$, with $z=(x,y)$, $x,y\in \R^n$. 
In the following proposition, we underline some basic properties of the operator $z^\ast$, see \cite[Lemma 3.1]{PSTV} for a proof.
\begin{proposition}\label{starop}
	The following properties hold:
	\begin{enumerate}
		\item[(i)] if $z_1,z_2\in\bbR^{2n}$ are linearly dependent, then $z_1\cdot z_2^*=0$;
		\item[(ii)] $z_1\cdot z_2=z_1^*\cdot z_2^*$ for each $z_1,z_2\in\bbR^{2n}$;
		\item[(iii)] if $\Omega\subset\R^{2n}$ is open and $f\in C^{\infty}(\Omega)$, then $\dive(\nabla f)^*=0$ on $\Omega$.
	\end{enumerate}
\end{proposition}

The following result, which generalizes \cite[Proposition 5.1]{PSTV}, states that if $\mathscr G_\Omega$ has a minimizer with some additional integrability, then it is unique.
\begin{proposition}\label{uniq-special}
	Let $p\in [1,2]$, let $\varphi \in W^{1,p'}(\Omega)$ and assume $g$ satisfies condition $(A)$. Let $u \in W^{1,p'}(\Omega)$ and $v \in W^{1,p}(\Omega)$ be two minimizers of
	\[
	\min\left\{\G_{\Omega}(u) : u \in \varphi+W^{1,p}_0(\Omega)\right\},
	\]
	then $u=v$ a.e.\,in $\Omega$.
\end{proposition}

\begin{proof}
	First of all we use a standard argument in order to prove that $\nabla u+\X^\ast$ and $\nabla v+\X^\ast$ are linearly dependent a.e.\,on $\Omega$. Using the convexity of $g$, we have
	\[
	g\left(\frac{\nabla u+\X^\ast}{2}+\frac{\nabla v+\X^\ast}{2}\right)\le \frac{g(\nabla u+\X^\ast)+g(\nabla v+\X^\ast)}{2} \quad \text{a.e.\,on $\Omega$}.
	\]
	Hence, from the minimality of $u$ and $v$ we get
	\[
	\G(u)\le \int_\Omega g\left(\frac{\nabla u+\X^\ast}{2}+\frac{\nabla v+\X^\ast}{2}\right) \ud\mathcal L^{2n}\le \frac{1}{2}\int_\Omega \left[g(\nabla u+\X^\ast)+g(\nabla v+\X^\ast)\right]\ud \mathcal L^{2n}=\G(u).
	\]
	Then
	\[
	g\left(\frac{\nabla u+\X^\ast}{2}+\frac{\nabla v+\X^\ast}{2}\right)= \frac{g(\nabla u+\X^\ast)+g(\nabla v+\X^\ast)}{2}, \quad \text{a.e.\,on $\Omega$}.
	\]
	Using (A) we deduce that $\nabla u+\X^\ast$ and $\nabla v+\X^\ast$ are linearly dependent a.e.\,on $\Omega$. The conclusion now follows proceeding exactly as in  the second part of \cite[Proposition 5.1]{PSTV}). \qedhere 
\end{proof}
\begin{remark}
	Notice that inequality \eqref{crescita} can be replaced by
	\begin{equation}\label{crescitaalt}
	\frac 1C|z|-C\leq g(z)\leq C(1+|z|),
	\end{equation}
	in which the map $g$ is not necessarily positive. This comes by the fact that, since we are studying minimizers, the function $g$ can be replaced by $g+M$, for any $M\in \R$.
\end{remark}
In order to prove the existence of a minimizer for $\G_{\Omega}$ we first compute its $L^1$ relaxed functional, namely
\begin{equation}
\mathcal{G}_{\Omega}(u):=\overline{\G_{\Omega}}(u)=\inf\left\{\liminf_h\G_{\Omega}(u_h) : u_h\in W^{1,1}(\Omega),\ u_h\to u\ \mbox{in}\ L^1(\Omega)\right\}.
\end{equation}
The following proposition provides an integral representation of $\mathcal{G}_{\Omega}$.
\begin{proposition}

Let $g$ be a convex function satisfying \eqref{crescita} and $u\in BV(\Omega)$. Then
\begin{equation}
\mathcal{G}_{\Omega}(u)=\int_{\Omega} g(\nabla u + \X^\ast)\,\ud\LL^{2n}+\int_\Omega g^\infty\left(\frac{dD^su}{d|D^su|}\right)\ud|D^su|.
\end{equation}
\end{proposition}
\begin{proof}
	By \cite[Remark 2.17]{FonsecaMuller}, it is enough to check (H1)-(H5) of the reference and observing that thanks to \eqref{remginf}, $g^{\infty}$ does not depend on $x$. Consider $f\colon\Omega\times \R^{2n}\to \R $ defined by $f(x,z)=g(z+X^*(x))$. Then $f$ is continuous and $f(x,\cdot)$ is convex. This yields (H1) and (H2).
	
	Assumption (H3) comes directly from \eqref{crescita} taking as $g(x,u)$ in (H3) the map $g(x,u)\equiv 1$. To prove (H4), we first recall that $g$ is Lipschitz with Lipschitz constant equal to $L_g$ and therefore
	\[
	|f(x,z)-f(x',z)|\leq L_g|x-x'|\leq L_g|x-x'|(1+|z|).
	\]
	In particular, if $x_0\in \Omega$ and $\delta>0$, then, whenever $|x-x_0|\leq \frac{\delta}{L_g}$ we get
	\[
	f(x,z)-f(x_0,z)\geq -L_g|x-x_0|(1+|z|)\geq -\delta(1+|z|),
	\]
	which completes the proof of (H4). Finally, (H5) comes from the fact that \eqref{crescita} implies
	\[
	|f^\infty(x,z)-f(x,z)|\leq C(|z|+1). \qedhere
	\]
	\end{proof}
Let $\Omega_0\subset \mathbb{R}^{2n}$ be an open Lipschitz domain with $\Omega\Subset \Omega_0$. Let $\varphi\in L^1(\partial\Omega)$ and $\Phi\in W^{1,1}(\Omega_0\setminus \Omega)$ such that $\Phi=\varphi$ on $\partial\Omega$ and $\Phi=0$ on $\partial \Omega_0$. We denote
\[
BV_{\Phi}(\Omega_0)\coloneqq \{u\in BV(\Omega_0):u=\Phi \ { \rm on } \  \Omega_0\setminus\overline \Omega\}
\]
by \cite[Theorem 1.3]{GMS} (see also \cite[Theorem $1.1$]{BF}) we get that $\argmin{\mathcal{G}_{\Omega_0}}\neq \emptyset$. Now observe that for any $u\in BV_{\Phi}(\Omega)$ we have
\begin{equation}\label{eas}
\begin{aligned}
\mathcal{G}_{\Omega_0}(u)&=\int_{\Omega} g(\nabla u + \X^\ast)\,\ud\LL^{2n}+\int_\Omega g^\infty\left(\frac{dD^su}{d|D^su|}\right)\ud|D^su|\\
&\qquad +\int_{\partial\Omega}g^\infty((\varphi-u_{|\partial\Omega})\nu_\Omega)\,\ud\mathcal H^{2n-1}+\mathcal{G}_{\Omega_0\setminus\Omega}(u_0)
\end{aligned}
\end{equation}
where $\nu_\Omega$ is the outer unit normal to $\Omega$ and $u_{|\partial\Omega}$ is the trace of $u$ on $\partial\Omega$.
Since the last term in the right-hand side of \eqref{eas} is constant we drop this term and we define the functional  $\mathcal{G}_{\varphi,\Omega} \colon BV(\Omega) \to \R$ by
\[
\mathcal{G}_{\varphi,\Omega}(u)\coloneqq\int_\Omega g(\nabla u+\X^\ast)\,\ud\LL^{2n}+\int_\Omega g^\infty\left(\frac{dD^su}{d|D^su|}\right)\ud|D^su|+\int_{\partial\Omega}g^\infty((\varphi-u_{|\partial\Omega})\nu_\Omega)\,\ud\mathcal H^{2n-1}
\]
whence for any $u\in BV_{\Phi}(\Omega)$,
\begin{equation}\label{xcv}
\mathcal{G}_{\Omega_0}(u)=\mathcal{G}_{\varphi,\Omega}(u_{|\Omega})+\mbox{constant}.
\end{equation}
Conversely, for any $u\in BV(\Omega)$ the extended function 
\[
u_0=\begin{cases}
u & \text{on $\Omega$},\\
\Phi & \text{on $\Omega_0\setminus \overline\Omega$} 
\end{cases}
\] 
 belongs to $BV_\Phi(\Omega_0)$ and 
\begin{equation*}
\mathcal{G}_{\Omega_0}(u_0)=\mathcal{G}_{\varphi,\Omega}(u)+\mbox{constant}.
\end{equation*}
We have then proved that, for any $\varphi\in L^1(\partial\Omega)$, the functional $\mathcal{G}_{\varphi,\Omega}$ admits a minimizers in $BV(\Omega)$.\\

The following result will be crucial later on, it relies on the approach developed in \cite{Giusti} for the area functional (see also \cite{BF}).
\begin{proposition} For any $\varphi\in L^1(\partial\Omega)$,
\begin{equation}
\min_{u\in BV(\Omega)} \mathcal{G}_{\varphi,\Omega}(u)=\inf\left\{\mathcal{G}_{\Omega}(u)\ :\ u\in W^{1,1}_0(\Omega)+\varphi\right\}.
\end{equation}
\end{proposition}
\begin{proof}
First we observe that $\mathcal{G}_{\Omega}(u)=\mathcal{G}_{\varphi,\Omega}(u)$ for any $u\in W^{1,1}_{0}(\Omega)+\varphi$, therefore
\begin{equation}
\inf\left\{\mathcal{G}_{\Omega}(u)\ :\ u\in W^{1,1}_0(\Omega)+\varphi\right\}\geq \min_{u\in BV(\Omega)} \mathcal{G}_{\varphi,\Omega}(u).
\end{equation}
Let $u\in BV(\Omega)$ and define $u_0\in BV_\Phi(\Omega_0)$ as above. Then by \cite[Lemma 2.1]{BF} there exists a sequence $(u_h)$ in $C^{\infty}_c(\Omega_0)$ such that $u_h=\Phi$ on $\Omega_0\setminus \overline\Omega$, 
$u_h\to u_0$ in $L^1(\Omega_0)$ and $\int_{\Omega_0}\sqrt{1+|\nabla u_h|^2}\to \int_{\Omega_0}\sqrt{1+|\nabla u_0|^2}$ as $h\to\infty$. Then, by Reshetnyak's continuity theorem (see e.g.\ \cite[Theorem 1.1]{Spec}) we get
\[
\mathcal{G}_{\Omega_0}(u_0)=\lim_{h} \mathcal{G}_{\Omega_0}(u_h)
\]
in particular
\begin{align*}
\mathcal{G}_{\varphi,\Omega}((u_0)_{|\Omega})=\lim_{h} \mathcal{G}_{\varphi, \Omega}((u_h)_{|\Omega})&=\lim_{h} \mathcal{G}_{\Omega}((u_h)_{|\Omega})\\
&\geq \inf\left\{\mathcal{G}_{\Omega}(u)\ :\ u\in W^{1,1}_0(\Omega)+\varphi\right\}
\end{align*}
and the conclusion follows.
\end{proof}

\subsection{A fundamental inequality} This subsection is devoted to proving the fundamental inequality \eqref{ineqfin}, which will be useful when dealing with comparison principles for minimizers of the functional $\mathcal G_\Omega$. This inequality is a generalization of the well known inequality for the perimeters that can be found, for the Euclidean case, in \cite[Proposition 3.38 (d)]{AFP} and has been extended for perimeters in the Heisenberg case in \cite{PSTV}. We underline also that, when dealing with Sobolev function with given boundary datum, this inequality turns out to be an equality whose proof is quite straightforward (see \cite[Lemma 5.1]{MThaar}).

\begin{theorem}\label{stimaMIN}
	Let $\Omega\subseteq \R^{2n}$ be an open and bounded set with Lipschitz boundary and let $g\colon \Omega\to [0,+\infty)$ be a convex function satisfying \eqref{crescita}. Then, for any $u_1,u_2\in BV(\Omega)$, we have
\begin{align}\label{ineqfin}
    \mathcal G_{\Omega}(u_1\vee u_2)+\mathcal G_{\Omega}(u_1\wedge u_2)\leq \mathcal G_{\Omega}(u_1)+\mathcal G_{\Omega}(u_2)\,.
\end{align}
\end{theorem}
\begin{proof}
Let us define
\begin{align*}
    &X\coloneqq \int_{\Omega}g\left(\nabla (u_1\vee u_2)+\X^*\right)\ud\LL^{2n}+\int_{\Omega}g\left(\nabla (u_1\wedge u_2)+\X^*\right)\ud\LL^{2n},\\
    &Y\coloneqq \int_{\Omega} g^\infty\left(\frac{dD^s(u_1\vee u_2)}{d|D^s(u_1\vee u_2)|}\right)\ud|D^c( u_1\vee u_2)|+\int_{\Omega} g^\infty\left(\frac{dD^s(u_1\wedge u_2)}{d|D^s(u_1\wedge u_2)|}\right)\ud|D^c(u_1\wedge u_2)|,\\
    &Z\coloneqq \int_{\Omega} g^\infty\left(\frac{dD^s(u_1\vee u_2)}{d|D^s(u_1\vee u_2)|}\right)\ud|D^j (u_1\vee u_2)|+\int_{\Omega} g^\infty\left(\frac{dD^s(u_1\wedge u_2)}{d|D^s(u_1\wedge u_2)|}\right)\ud|D^j (u_1\wedge u_2)|\,.
\end{align*}
Observe that \eqref{ineqfin} will follow if we show that
\begin{align}\label{stimfond}
    X+Y+Z\leq \mathcal G_{\Omega}(u_1)+\mathcal G_{\Omega}(u_2).
\end{align}

Without loss of generality, we may assume that $u_1=\tilde u_1$ on $\Omega\setminus S_{u_1}$ and $u_2=\tilde u_2$ on
$\Omega\setminus S_{u_2}$. Setting
\[
\Omega_+\coloneqq (\Omega\setminus (S_{u_1}\cup S_{u_2}))\cap \{u_1\geq u_2\},\quad \Omega_-\coloneqq(\Omega\setminus (S_{u_1}\cup S_{u_2}))\cap \{u_1 < u_2\}
\]
we have (see e.g.\ \cite[Example 3.100]{AFP})
\[
\begin{split}
& \nabla (u_1\vee u_2) = \nabla u_1\,\chi_{\Omega_+} + \nabla u_2\,\chi_{\Omega_-}\quad\LL^{2n}\text{-a.e. in }\Omega\\
& \nabla (u_1\wedge u_2) = \nabla u_2\,\chi_{\Omega_+} + \nabla
u_1\,\chi_{\Omega_-}\quad\LL^{2n}\text{-a.e. in }\Omega\,,
\end{split}
\]
where $\chi_E$ denotes the characteristic function of a set $E$, and similarly
\[
\begin{split}
& D^c (u_1\vee u_2) = D^c u_1\res{\Omega_+} + D^c u_2\res{\Omega_-};\\
& D^c (u_1\wedge u_2) = D^c u_2\res{\Omega_+} + D^c
u_1\res{\Omega_-}.
\end{split}
\]
Therefore
\begin{equation}\label{stima1}
\begin{aligned}
    X=&\int_{\Omega_+}g\left(\nabla u_1+\X^{\ast}\right)\ud\LL^{2n}+\int_{\Omega_-}g\left(\nabla u_2+\X^{\ast}\right)\ud\LL^{2n}\\&+\int_{\Omega_+}g\left(\nabla u_2+\X^{\ast}\right)\ud\LL^{2n}+\int_{\Omega_-}g\left(\nabla u_1+\X^{\ast}\right)\ud\LL^{2n}\\
    =&\int_{\Omega\setminus (S_{u_1}\cup S_{u_2}) }g\left(\nabla u_1+\X^{\ast}\right)\ud\LL^{2n}+\int_{\Omega\setminus (S_{u_1}\cup S_{u_2})}g\left(\nabla u_2+\X^{\ast}\right)\ud\LL^{2n}\\
    =&\int_{\Omega}g\left(\nabla u_1+\X^{\ast}\right)\ud\LL^{2n}+\int_{\Omega}g\left(\nabla u_2+\X^{\ast}\right)\ud\LL^{2n}.
\end{aligned}
\end{equation}
and
\begin{equation}\label{eq:stima2}
\begin{aligned}
Y & =\int_{\Omega_+}g^\infty\left(\frac{dD^cu_1}{d|D^cu_1|}\right)\ud|D^cu_1|+\int_{\Omega_-}g^\infty\left(\frac{dD^cu_2}{d|D^cu_2|}\right)\ud|D^cu_2|\\&\hphantom{=}+\int_{\Omega_+}g^\infty\left(\frac{dD^cu_2}{d|D^cu_2|}\right)\ud|D^cu_2|+\int_{\Omega_-}g^\infty\left(\frac{dD^cu_1}{d|D^cu_1|}\right)\ud|D^cu_1|\\&=\int_{\Omega\setminus(S_{u_1}\cup S_{u_2})}g^\infty\left(\frac{dD^su_1}{d|D^su_1|}\right)\ud|D^cu_1|+\int_{\Omega\setminus(S_{u_1}\cup S_{u_2})}g^\infty\left(\frac{dD^su_2}{d|D^su_2|}\right)\ud|D^cu_2|\\
&=\int_{\Omega}g^\infty\left(\frac{dD^su_1}{d|D^su_1|}\right)\ud|D^cu_1|+\int_{\Omega}g^\infty\left(\frac{dD^su_2}{d|D^su_2|}\right)\ud|D^cu_2|,
\end{aligned}
\end{equation}
where to obtain the last equality in \eqref{stima1} and in \eqref{eq:stima2}, we used the fact that $\mathcal{L}^{2n}(S_{u_1}\cup S_{u_2})=0$ ( see \cite[Proposition 3.64]{AFP}) and the fact that, since $u,v\in BV(\Omega)$, then $|D^cu|(S_v)=|D^cv|(S_u)=0$.\footnote{This last fact follows from Proposition $3.92$ item c) and Remark $2.50$ in \cite{AFP}}

Recall that, by \cite[Eq.\ (3.90)]{AFP}, one has
\[
\begin{split}
    &D^j u_1=(u_1^+-u_1^-)\nu_1 \HH^{2n-1}\res J_{u_1}\\
    &D^j u_2=(u_2^+-u_2^-)\nu_2 \HH^{2n-1}\res J_{u_2},
\end{split}
\]
where $\nu_1,\nu_2$ are the unit normals to the $(2n-1)$-rectifiable sets $J_{u_1},J_{u_2}$. Without loss of generality, we may assume that $u_1^+\geq u_1^-$ and  $\nu_1=\nu_2$, $\HH^{2n-1}$-a.e. on $J_{u_1}\cap J_{u_2}$; in this way, the $(2n-1)$-rectifiable set $T\coloneqq J_{u_1}\cup J_{u_2}$ is associated with the unit normal $\nu_T$ defined by
\[
\nu_T\coloneqq\nu_1 \text{ on }J_{u_1},\qquad\nu_T\coloneqq\nu_2 \text{ on }T\setminus J_{u_1}.
\]
We extend $u_1^\pm\colon J_{u_1}\to\R$ and $u_2^\pm\colon J_{u_2}\to\R$ to the whole $T$ by setting
\[
u_1^\pm\coloneqq\begin{cases}
u_1^\pm & \text{on }J_{u_1}\\
0 & \text{on }T\setminus J_{u_1},
\end{cases}
\qquad
u_2^\pm\coloneqq\begin{cases}
u_2^\pm & \text{on }J_{u_2}\\
0 & \text{on }T\setminus J_{u_2}.
\end{cases}
\]
In this way one has
\[
D^j(u_1+u_2) = (u_1^+ - u_1^- + u_2^+  - u_2^-)\,\nu_T\:\HH^{2n-1}\res T\,.
\]
By \cite[Theorem 3.99]{AFP}, $|u_1-u_2|\in BV(\Omega)$ and
\begin{equation}\label{jumpg}
D^j (|u_1-u_2|) =  (|u_1^+-u_2^+| - |u_1^- - u_2^-|)\,\nu_T\: \HH^{2n-1}\res T.
\end{equation}
We can then write
\[
\begin{split}
    &D^j (u_1\vee u_2)=D^j\left(\tfrac{u_1+u_2}{2} + \tfrac{|u_1-u_2|}{2}\right) =\tfrac 12D^j\left(u_1+u_2\right)+\tfrac12 D^j\left( |u_1-u_2|\right)\\
    &D^j (u_1\wedge u_2)=D^j\left(\tfrac{u_1+u_2}{2} - \tfrac{|u_1-u_2|}{2}\right)= \tfrac 12D^j\big(u_1+u_2\big)-\tfrac 12D^j\left( |u_1-u_2|\right).
\end{split}
\]
By using this decomposition and \eqref{jumpg}, we have
\begin{equation}\label{eq:stimaC}
\begin{aligned}
    Z=&\int_{T}g^\infty\left(\frac{dD^j(u_1\vee u_2)}{d|D^j(u_1\vee u_2)|}\right)\ud|D^j (u_1\vee u_2)|+\int_{T}g^\infty\left(\frac{dD^j(u_1\wedge u_2)}{d|D^j(u_1\wedge u_2)|}\right)\ud|D^j (u_1\wedge u_2)|\\
    =&\frac{1}{2}\int_T g^\infty\left(\left(u_1^+-u_1^-+u_2^+-u_2^-+|u_1^+-u_2^+|-|u_1^--u_2^-|\right)\nu_T\right)\ud\HH^{2n-1}\\
    &+\frac{1}{2}\int_T g^{\infty}\left(\left(u_1^+-u_1^-+u_2^+-u_2^--|u_1^+-u_2^+|+|u_1^--u_2^-|\right) \nu_T\right)\ud\HH^{2n-1}.
\end{aligned}
\end{equation}
Let for shortness $\alpha, \beta\colon T\to \R$ be the functions defined by
\[
\begin{aligned}
\alpha&\coloneqq u_1^+-u_1^-+u_2^+-u_2^-+|u_1^+-u_2^+|-|u_1^--u_2^-|,\\
\beta&\coloneqq u_1^+-u_1^-+u_2^+-u_2^--|u_1^+-u_2^+|+|u_1^--u_2^-|.
\end{aligned}
\]
To estimate $Z$, we are going to split $T$ into several regions. Set
\[
T'\coloneqq\{x\in T: u_2^+(x)\geq u_2^-(x)\}, \quad \text{and}\quad T''\coloneqq\{x\in T:u_2^+(x)< u_2^-(x) \}.
\]
Then, taking into account that $u_1^-\leq u_1^+$ on $T$, one can easily check that both $\alpha$ and $\beta$ are positive on $T'$. Being $g^\infty$ positively homogeneous, then one has
\begin{equation}\label{eq:identitasuT'}
\begin{aligned}
&\frac 12\int_{T'}g^\infty(\alpha \nu_T)\ud\HH^{2n-1}+ \frac 12\int_{T'}g^\infty(\beta \nu_T)\ud\HH^{2n-1}=\frac{1}{2}\int_{T'}(\alpha+\beta)g^\infty(\nu_T)\ud\HH^{2n-1}\\
& =\int_{T'}(u_1^+-u_1^-)g^\infty(\nu_T)\ud\HH^{2n-1}+\int_{T'}(u_2^+-u_2^-)g^\infty(\nu_T)\ud\HH^{2n-1}\\
&= \int_{T'}g^\infty((u_1^+-u_1^-)\nu_T)\ud\HH^{2n-1}+\int_{T'}g^\infty((u_2^+-u_2^-)\nu_T)\ud\HH^{2n-1}.
\end{aligned}
\end{equation}
We now subdivide $T''$ into the union of the following disjoint subsets:
\[
\begin{aligned}
T''_{++}&\coloneqq\{x\in T'':u_1^+(x)\geq u_2^+(x), u_1^-(x)\geq u_2^-(x) \}, \quad T''_{--}\coloneqq \{x\in T'':u_1^+(x)< u_2^+(x), u_1^-(x)< u_2^-(x) \}\\
T''_{+-}&\coloneqq \{x\in T'':u_1^+(x)\geq u_2^+(x), u_1^-(x)< u_2^-(x) \}, \quad T''_{-+}\coloneqq\{x\in T'':u_1^+(x)< u_2^+(x), u_1^-(x)\geq u_2^-(x) \}.
\end{aligned}
\]
Notice that, for every $x\in T''_{++}$, one has $\alpha(x)=2(u_1^+(x)-u_1^-(x))$ and $\beta(x)=2(u_2^+(x)-u_2^-(x))$,  conversely, for every $x\in T''_{--}$, one has $\alpha(x)=2(u_2^+(x)-u_2^-(x))$ and $\beta(x)=2(u_1^+(x)-u_1^-(x))$. Using this information, we easily obtain
\begin{equation}\label{eq:stima++--}
\begin{aligned}
&\frac 12\int_{T''_{++}\cup T''_{--}}g^{\infty}(\alpha \nu_T)\ud \HH^{2n-1}+\frac 12\int_{T''_{++}\cup T''_{--}}g^{\infty}(\beta \nu_T)\ud \HH^{2n-1}\\&=\int_{T''_{++}\cup T''_{--}}g^{\infty}((u_1^+-u_1^-) \nu_T)\ud \HH^{2n-1}+\int_{T''_{++}\cup T''_{--}}g^{\infty}((u_2^+-u_2^-) \nu_T)\ud \HH^{2n-1}.
\end{aligned}
\end{equation}
We now consider $T''_{+-}$. The estimate on $T''_{-+}$ can be done in a completely analogous way. We first write $T''_{+-}=\Gamma_1\cup\Gamma_2\cup\Gamma_3\cup \Gamma_4$, where
\[
\begin{aligned}
\Gamma_1&\coloneqq\{x\in T''_{+-}: u_1^+(x)\geq u_2^-(x), u_2^+(x)\geq u_1^-(x) \}, \quad \Gamma_2\coloneqq\{x\in T''_{+-}: u_1^+(x)\geq u_2^-(x), u_2^+(x)< u_1^-(x) \}\\
\Gamma_3&\coloneqq\{x\in T''_{+-}: u_1^+(x)< u_2^-(x), u_2^+(x)\geq u_1^-(x) \},\quad \Gamma_4\coloneqq\{x\in T''_{+-}: u_1^+(x)< u_2^-(x), u_2^+(x)<u_1^-(x) \}.
\end{aligned}
\]
Notice that, for every $x\in T''_{+-}$, one has that $\alpha(x)=2(u_1^+(x)-u_2^-(x))$ and $\beta(x)=2(u_2^+(x)-u_1^-(x))$ and, by construction, $\alpha$ is positive on $\Gamma_1\cup \Gamma_2$ and strictly negative on $\Gamma_3\cup \Gamma_4$, while $\beta$ is positive on $\Gamma_1\cup\Gamma_3$ and strictly negative on $\Gamma_2\cup \Gamma_4$. Using the positive homogeneity of $g^\infty$, we get
\begin{equation}\label{eq:stimaGamma1}
\begin{aligned}
&\frac 12\int_{\Gamma_1} g^\infty(\alpha \nu_T)\ud\HH^{2n-1}+\frac 12\int_{\Gamma_1} g^\infty(\beta \nu_T)\ud\HH^{2n-1}\\&=\int_{\Gamma_1}(u_1^+-u_2^-)g^\infty(\nu_T)\ud\HH^{2n-1}+\int_{\Gamma_1}(u_2^+-u_1^-)g^\infty(\nu_T)\ud\HH^{2n-1}\\&=\int_{\Gamma_1}(u_1^+-u_1^-)g^\infty(\nu_T)\ud\HH^{2n-1}+\int_{\Gamma_1}(u_2^+-u_2^-)g^\infty(\nu_T)\ud\HH^{2n-1}\\& =\int_{\Gamma_1}g^\infty((u_1^+-u_1^-)\nu_T)\ud\HH^{2n-1}+\int_{\Gamma_1}g^\infty((u_2^+-u_2^-)\nu_T)\ud\HH^{2n-1}.
\end{aligned}
\end{equation}
Taking into account that $\alpha$ and $\beta$ are strictly negative on $\Gamma_4$, we also have
\begin{equation}\label{eq:stimaGamma4}
\begin{aligned}
&\frac 12\int_{\Gamma_4} g^\infty(\alpha \nu_T)\ud\HH^{2n-1}+\frac 12\int_{\Gamma_4} g^\infty(\beta \nu_T)\ud\HH^{2n-1}\\&=\int_{\Gamma_4}(u_2^--u_1^+)g^\infty(-\nu_T)\ud\HH^{2n-1}+\int_{\Gamma_4}(u_1^--u_2^+)g^\infty(-\nu_T)\ud\HH^{2n-1}\\&=\int_{\Gamma_4}(u_1^--u_1^+)g^\infty(-\nu_T)\ud\HH^{2n-1}+\int_{\Gamma_4}(u_2^--u_2^+)g^\infty(-\nu_T)\ud\HH^{2n-1}\\& =\int_{\Gamma_4}g^\infty((u_1^+-u_1^-)\nu_T)\ud\HH^{2n-1}+\int_{\Gamma_4}g^\infty((u_2^+-u_2^-)\nu_T)\ud\HH^{2n-1}.
\end{aligned}
\end{equation}
Recall that, by \eqref{crescita}, the map $g^\infty$ is positive, and therefore, for any $0\leq \lambda_1\leq \lambda_2$ and any $x\in \R^{2n}$, one has $g^{\infty}(\lambda_1x)\leq g^\infty(\lambda_2x)$. We can make the estimate on $\Gamma_2$, taking into account that $\alpha$ is positive and $\beta$ is strictly negative:
\begin{equation}\label{eq:stimaGamma2}
\begin{aligned}
&\frac 12\int_{\Gamma_2} g^\infty(\alpha \nu_T)\ud\HH^{2n-1}+\frac 12\int_{\Gamma_2} g^\infty(\beta \nu_T)\ud\HH^{2n-1}\\&=\int_{\Gamma_2}(u_1^+-u_2^-)g^\infty(\nu_T)\ud\HH^{2n-1}+\int_{\Gamma_2}(u_1^--u_2^+)g^\infty(-\nu_T)\ud\HH^{2n-1}\\&=\int_{\Gamma_2}(u_1^+-u_1^-+u_1^--u_2^-)g^\infty(\nu_T)\ud\HH^{2n-1}+\int_{\Gamma_2}(u_1^--u_2^-+u_2^--u_2^+)g^\infty(-\nu_T)\ud\HH^{2n-1}\\& \leq\int_{\Gamma_2}g^\infty((u_1^+-u_1^-)\nu_T)\ud\HH^{2n-1}+\int_{\Gamma_2}g^\infty((u_2^+-u_2^-)\nu_T)\ud\HH^{2n-1},
\end{aligned}
\end{equation}
where in the last inequality we used the fact that $(u_1^--u_2^-)_{|T''_{+-}}<0$ and $(u_2^--u_2^+)_{|T''}>0$. Analogously, for $\Gamma_3$, we have
\begin{equation}\label{eq:stimaGamma3}
\begin{aligned}
&\frac 12\int_{\Gamma_3} g^\infty(\alpha \nu_T)\ud\HH^{2n-1}+\frac 12\int_{\Gamma_3} g^\infty(\beta \nu_T)\ud\HH^{2n-1}\\&=\int_{\Gamma_3}(u_2^--u_1^+)g^\infty(-\nu_T)\ud\HH^{2n-1}+\int_{\Gamma_3}(u_2^+-u_1^-)g^\infty(\nu_T)\ud\HH^{2n-1}\\&=\int_{\Gamma_3}(u_2^--u_2^++u_2^+-u_1^+)g^\infty(-\nu_T)\ud\HH^{2n-1}+\int_{\Gamma_3}(u_2^+-u_1^++u_1^+-u_1^-)g^\infty(\nu_T)\ud\HH^{2n-1}\\& \leq\int_{\Gamma_3}g^\infty((u_2^+-u_2^-)\nu_T)\ud\HH^{2n-1}+\int_{\Gamma_3}g^\infty((u_1^+-u_1^-)\nu_T)\ud\HH^{2n-1},
\end{aligned}
\end{equation}
where in the last inequality we have used the fact that $(u_2^--u_2^+)_{|T''}>0$ and $(u_2^+-u_1^+)_{|T''_{+-}}\leq0$. Combining \eqref{eq:stimaGamma1}, \eqref{eq:stimaGamma2}, \eqref{eq:stimaGamma3} and \eqref{eq:stimaGamma4} one obtains
\begin{equation}\label{eq:stimaT+-}
\begin{aligned}
&\frac 12\int_{T''_{+-}} g^\infty(\alpha \nu_T)\ud\HH^{2n-1}+\frac 12\int_{T''_{+-}} g^\infty(\beta \nu_T)\ud\HH^{2n-1} \\ &\leq\int_{T''_{+-}}g^\infty((u_2^+-u_2^-)\nu_T)\ud\HH^{2n-1}+\int_{T''_{+-}}g^\infty((u_1^+-u_1^-)\nu_T)\ud\HH^{2n-1}.
\end{aligned}
\end{equation}
In a completely analogous fashion, we can also write
\begin{equation}\label{eq:stimaT-+}
\begin{aligned}
&\frac 12\int_{T''_{-+}} g^\infty(\alpha \nu_T)\ud\HH^{2n-1}+\frac 12\int_{T''_{-+}} g^\infty(\beta \nu_T)\ud\HH^{2n-1} \\ &\leq\int_{T''_{-+}}g^\infty((u_2^+-u_2^-)\nu_T)\ud\HH^{2n-1}+\int_{T''_{-+}}g^\infty((u_1^+-u_1^-)\nu_T)\ud\HH^{2n-1}.
\end{aligned}
\end{equation}
As a direct consequence of \eqref{eq:stima++--}, \eqref{eq:stimaT+-} and \eqref{eq:stimaT-+}, we then have
\begin{equation}\label{eq:stimaT''}
\begin{aligned}
&\frac 12\int_{T''} g^\infty(\alpha \nu_T)\ud\HH^{2n-1}+\frac 12\int_{T''} g^\infty(\beta \nu_T)\ud\HH^{2n-1} \\ &\leq\int_{T''}g^\infty((u_2^+-u_2^-)\nu_T)\ud\HH^{2n-1}+\int_{T''}g^\infty((u_1^+-u_1^-)\nu_T)\ud\HH^{2n-1}.
\end{aligned}
\end{equation}
The thesis is then obtained by combining \eqref{stima1}, \eqref{eq:stima2}, \eqref{eq:stimaC}, \eqref{eq:identitasuT'} and \eqref{eq:stimaT''}.
\end{proof}
\begin{corollary}\label{lem:disbordo}
	Let $\Omega\subseteq \R^{2n}$ be an open and bounded set with Lipschitz boundary and let $g\colon \Omega\to [0,+\infty)$ be a convex function satisfying \eqref{crescita}.
	Then, for every $\varphi_1, \varphi_2\in L^1(\partial\Omega)$ and every $u_1,u_2\in BV(\Omega)$ one has
	\begin{equation}\label{eq:ineqfincondato}
	\mathcal G_{\Omega, \varphi_1\vee \varphi_2}(u_1\vee u_2)+\mathcal G_{\Omega, \varphi_1\wedge \varphi_2}(u_1\wedge u_2)\leq \mathcal G_{\Omega, \varphi_1}(u_1)+\mathcal G_{\Omega, \varphi_2}( u_2).
	\end{equation}
\end{corollary}
\begin{proof}
	Let $u_1, u_2\in BV(\Omega)$ and $\varphi_1,\varphi_2\in L^1(\partial \Omega)$. Notice that, if $u_1,u_2 \in W^{1,1}(\Omega)$, we immediately have
	\begin{equation}\label{eq:ugualeSobolev}
	\mathcal G_\Omega(u_1\vee u_2)+\mathcal G_\Omega(u_1\wedge u_2)= \mathcal G_\Omega(u_1)+\mathcal G_\Omega(u_2).
	\end{equation}
	Fix any bounded open and Lipschitz set $\Omega_0\Supset\Omega$. By \cite[Theorem 2.16]{Giusti}, we can find $w_1,w_2\in W^{1,1}(\Omega_0\setminus \overline{\Omega})$ with ${w_1}_{|\partial \Omega}=\varphi_1$ and ${w_2}_{|\partial \Omega}=\varphi_2$. Set now
	\[
	v_1\coloneqq
	\begin{cases}
	w_1 & \text{ on $\Omega_0\setminus \overline{\Omega}$}\\
	u_1&	\text{ on $\Omega$}
	\end{cases}\quad \text{ and }\quad v_2\coloneqq\begin{cases}
	w_2 &\text{ on $\Omega_0\setminus \overline{\Omega}$}\\
	u_2&	\text{ on $\Omega$}.
	\end{cases}
	\]
	By \cite[Theorem 3.84]{AFP}, $v_1,v_2\in BV(\Omega_0)$ and, moreover, if $\nu_\Omega$ denotes the exterior normal to $\Omega$,  one has
	\[
	Dv_i=Du_i\res \Omega+ Dw_i \res (\Omega_0\setminus \overline{\Omega})+ (w_i-u_i) \nu_\Omega \mathcal H^{2n-1}\res \partial \Omega, \quad \text{ for }\; i=1,2,
	\]
	from which we can compute, up to $|D^sv|$-negligible sets, the polar vector:
	\[
	\frac{dD^sv_i}{d|D^sv_i|}=\begin{cases}
	\displaystyle\frac{dD^su_i}{d|D^su_i|} & \text{ on $\Omega$}\\
	0& \text{ on $\Omega_0\setminus \overline \Omega$}\\
	\displaystyle \frac{(w_i-u_i) }{|w_i-u_i|}\nu_\Omega &\text{ on $\partial\Omega$}.
	\end{cases}
	\]
	Then, using the previous expression, the fact that $g^\infty$ is homogeneous and the definition of $w_i$, we get
	\[
	\begin{aligned}
	\mathcal G_{\Omega_0}(v_i)&=\int_{\Omega_0}g(\nabla v_i+\X^{\ast})\ud\mathcal L^{2n}+\int_{\Omega_0} g^\infty\left(\frac{dD^sv_i}{d|D^sv_i|}\right)\ud|D^sv_i|\\
	&= \mathcal G_{\Omega}(u_i)+ \mathcal G_{\Omega_0\setminus \overline \Omega}(w_i)+\int_{\partial \Omega} g^\infty((w_i-u_i)\otimes\nu_\Omega)\ud\mathcal H^{2n-1}\\&= \mathcal G_{\Omega}(u_i)+\mathcal G_{\Omega_0\setminus \overline \Omega}(w_i)+\int_{\partial \Omega} g^\infty((\varphi_i-u_i)\otimes\nu_\Omega)\ud\mathcal H^{2n-1}\\&= \mathcal G_{\Omega_0 \setminus \overline\Omega}(w_i)+ \mathcal G_{\Omega, \varphi_i}(u_i), \quad \text{ for $i=1,2$}.
	\end{aligned}
	\]
	Similarly, using Remark \ref{traccesupinf} we also have
	\[
	\begin{aligned}
	\mathcal G_{\Omega_0}(v_1\vee v_2)&= \mathcal G_{\Omega_0\setminus\overline \Omega}(w_1\vee w_2)+ \mathcal G_{\Omega, \varphi_1\vee \varphi _2}(u_1\vee u_2) \quad { and}\\
	\mathcal G_{\Omega_0}(v_1\wedge v_2)&= \mathcal G_{\Omega_0\setminus\overline \Omega}(w_1\wedge w_2)+ \mathcal G_{\Omega, \varphi_1\wedge \varphi _2}(u_1\wedge u_2).
	\end{aligned}
	\]
	We can then conclude combining the previous identities with \eqref{eq:ugualeSobolev}  applied on $\Omega_0\setminus \overline{\Omega}$, $w_1, w_2$ and Theorem \ref{stimaMIN} applied on $\Omega_0$, $v_1$ and $v_2$:
	\[
	\begin{split}
	\mathcal G_{\Omega, \varphi_1\vee \varphi_2}(u_1\vee u_2) + \mathcal G_{\Omega, \varphi_1\wedge \varphi_2}(u_1\wedge u_2)\\=\mathcal G_{\Omega_0}(v_1\vee v_2) +\mathcal G_{\Omega_0}(v_1\wedge v_2)-\mathcal G_{\Omega_0\setminus\overline\Omega}( w_1\vee w_2) -\mathcal G_{\Omega_0\setminus\overline\Omega}( w_1\wedge w_2)\\
	\leq \mathcal G_{\Omega_0}(v_1)+\mathcal G_{\Omega_0}(v_2)-\mathcal G_{\Omega_0\setminus\overline\Omega}(w_1)-\mathcal G_{\Omega_0\setminus\overline\Omega}(w_2)=\mathcal G_{\Omega, \varphi_1}(u_1)+\mathcal G_{\Omega,\varphi_2}(u_2).\qedhere
	\end{split}
	\]
\end{proof}


\subsection{The set of minimizers and comparison principles}

Given a  bounded open set $\Omega\subset\bbR^{2n}$ with Lipschitz regular boundary and a function $\varphi\in L^1(\partial\Omega)$ we define
\[
\cc_\varphi\coloneqq\argmin\limits_{u} \area_{\varphi,\Omega}(u)\,.
\]
We have already proved that $\cc_\varphi\subset \BV(\Omega)$ is nonempty.  \\
Using Theorem \ref{stimaMIN} and Corollary \ref{lem:disbordo}, the proof of Proposition \ref{maxmin} below is completely analogous to \cite[Proposition 4.3]{PSTV} and we omit it.

\begin{proposition}\label{maxmin}
	Let $\varphi_1,\varphi_2 \in L^1(\partial\Omega)$ be such that
	$\varphi_1\leq \varphi_2$ $\HH^{2n-1}$-a.e.\ on $\partial\Omega$ and assume
	that $u_1\in\cc_{\varphi_1}$ and $u_2\in\cc_{\varphi_2}$. Then
	$(u_1\vee u_2)\in\cc_{\varphi_2}$ and  $(u_1\wedge
	u_2)\in\cc_{\varphi_1}$.
\end{proposition}

In \cite{MTccm} (see also \cite{PSTV}), it has been proved that the set of minimizers of a superlinear convex functional has a maximum $u$ (resp.\ a minimum $u$) defined as the pointwise supremum (infimum) of the minimizers. These special minimizers are then used to prove one-sided Comparison Principles. 

\begin{proposition}\label{esistmax}
	Let $\Omega\subset\bbR^{2n}$ be a bounded open set with Lipschitz regular boundary and let $\varphi\in L^1(\partial\Omega)$. Then, there
	exists $\overline{u},\underline{u}\in \cc_\varphi$ such that the
	inequalities
	\begin{equation}\label{min9}
	\underline u \leq u\leq \overline u,\quad \LL^{2n}\mbox{-a.e. in}\  \Omega
	\end{equation}
	hold for any $u\in\cc_\varphi$.
\end{proposition}
\begin{proof}
	We start by proving that $\cc_\varphi$ is bounded in $\BV(\Omega)$. Define $J\coloneqq \min_{u\in BV(\Omega)}\area_{\varphi,\Omega}(u)<+\infty$. By \eqref{crescitainfty} and denoting by $\tilde{C}=\sup_{\Omega} |\X^{\ast}|$ we get
	\begin{equation}\label{stimavar}
	\begin{split}
	|D u|(\Omega) = & \int_\Omega |\nabla u|\ud\LL^{2n} + |D^su|(\Omega)\\
	\leq & C\int_{\Omega}g(\nabla u+\X^{\ast})\ud\LL^{2n} +\tilde C|\Omega| +\int_\Omega \ud |D^su|\\
	\leq& C\int_\Omega g(\nabla u+\X^{\ast})\ud\LL^{2n}+\tilde C|\Omega|+C\int_\Omega g^\infty\left(\frac{dD^su}{d|D^su|}\right)\ud|D^su|+C\int_{\partial \Omega} g^\infty((\varphi-u_{|\partial\Omega})\nu_\Omega)\ud\mathcal H^{2n-1}\\=& CJ+\tilde C|\Omega|, \qquad\forall u\in \cc_\varphi,
	\end{split}
	\end{equation}
	where $|\Omega|\coloneqq\LL^{2n}(\Omega)$.
	Moreover, by \cite[Theorem 1.28 and Remark 2.14]{Giusti} there exists $c=c(n)>0$ such that
	\[
	\begin{split}
	\norm{u}_{L^1(\Omega)} &\leq |\Omega|^{1/2n}\norm{u}_{L^{2n/(2n-1)}(\Omega)}\\
	&\leq  c\,|\Omega|^{1/2n}\left( |D u|(\Omega)+\int_{\partial\Omega}|u|    \ud\HH^{2n-1}\right)\\
	&\leq   c|\Omega|^{1/2n}\left( |D u|(\Omega)+\int_{\partial\Omega}|\varphi-u_{|\partial\Omega}|\ud\HH^{2n-1}  +\int_{\partial\Omega}|\varphi|\ud\HH^{2n-1}\right)\\
	&= c|\Omega|^{1/2n}\left( |Du|(\Omega)+\int_{\partial\Omega}|(\varphi -u_{|\partial\Omega})\nu_\Omega|\ud\HH^{2n-1}+\int_{\partial\Omega}|\varphi|\ud\HH^{2n-1}\right)\\
	&\leq c|\Omega|^{1/2n}\left( C\int_\Omega g(\nabla u+X^*)\ud\LL^{2n}+\tilde C|\Omega|+C\int_\Omega g^\infty\left(\frac{dD^su}{d|D^su|}\right)|D^su|(\Omega)\right. \\
	& \left.\hphantom{= c|\Omega|^{1/2n}}+C\int_{\partial\Omega}g^\infty((\varphi-u_{|\partial \Omega})\nu_\Omega)\ud \HH^{2n-1}+ \int_{\partial\Omega}|\varphi|\ud\HH^{2n-1}\right)\\
	&=  c|\Omega|^{1/2n}\left(CJ+\tilde C|\Omega|+\int_{\partial\Omega}|\varphi|\ud\HH^{2n-1}\right),\qquad \forall u\in \cc_\varphi,
	\end{split}
	\]
	where in the second last inequality we argued as in \eqref{stimavar}.
	This, together with \eqref{stimavar}, implies that $\cc_\varphi$ is bounded in $\BV(\Omega)$.
	
	Therefore, by \cite[Theorem 3.23]{AFP}, $\cc_\varphi$ is pre-compact in $L^1(\Omega)$, i.e., for every sequence $(u_h)$ in  $\cc_\varphi$ there exist $u\in \BV(\Omega)$ and a subsequence $(u_{h_k})$ such that $u_{h_k}\to u$ in $L^1(\Omega)$. By \eqref{xcv}, $\area_{\varphi,\Omega}$ is lower semicontinuous with respect to the $L^1$-convergence, hence we have also
	\[
	\area_{\varphi,\Omega}(u)\leq\liminf_{k\to\infty} \area_{\varphi,\Omega}(u_{h_k})=J,
	\]
	so that $u\in\cc_\varphi$. We have proved that $\cc_\varphi$ is compact in $L^1(\Omega)$. Now, the functional
	\[
	\BV(\Omega)\ni u\longmapsto I(u)\coloneqq\int_{\Omega} u\ \ud\LL^{2n}
	\]
	is continuous in $L^1(\Omega)$, hence it admits maximum $\overline{u}$ and minimum $\underline u$ in $\cc_\varphi$:
	let us prove that $\overline u,\underline u$ satisfy \eqref{min9} for any  $u\in \cc_\varphi$.
	
	Assume by contradiction there exists $u\in\cc_\varphi$ such that $\Omega'\coloneqq\{z\in\Omega: u(z)>\overline{u}(z)\}$ has strictly positive measure. Then, by Corollary \ref{maxmin}, $u\vee\overline u$ is in $\cc_\varphi$.
	Moreover
	\[
	\int_\Omega (u\vee\overline u)\ud\LL^{2n}=\int_{\Omega'} u\ud\LL^{2n}+\int_{\Omega\setminus\Omega'} \overline u\ud\LL^{2n}>\int_\Omega \overline u\ud\LL^{2n}
	\]
	yielding a contradiction. The fact that $u\geq \underline u$ follows in a similar way.
\end{proof}
The following result is a Comparison Principle inspired by the results obtained in \cite{MTccm} for superlinear functionals in Sobolev spaces and it can be proved exactly as in \cite[Theorem 4.5]{PSTV}.
\begin{theorem}\label{confro}
	Let $\Omega\subset\bbR^{2n}$ be a bounded open set with Lipschitz regular
	boundary; let $\varphi,\psi\in L^1(\partial\Omega)$ be such that
	$\varphi\leq \psi$ $\HH^{2n-1}$-a.e. on $\partial\Omega$. Consider
	the functions $\overline u,\,\underline u\in \cc_{\varphi}$ and
	$\overline w,\, \underline w\in \cc_{\psi}$ such that\footnote{The existence of $\overline u,\,\underline u,\,\overline w,\,\underline w$ is guaranteed by Proposition \ref{esistmax}.}
	\begin{equation}\label{uuww}
	\begin{array}{ll}
	\underline u \leq u\leq \overline u\qquad &\LL^{2n}\mbox{-a.e. in }\Omega,\ \forall u\in\cc_\varphi\\
	\underline w \leq w\leq \overline w\quad &\LL^{2n}\mbox{-a.e. in }\Omega,\ \forall w\in\cc_\psi\,.
	\end{array}
	\end{equation}
	Then
	\begin{equation}\label{AM}
	\overline u \leq\overline w\quad\text{and}\quad\underline
	u\leq\underline w\quad \LL^{2n}\mbox{-a.e. in}\ \Omega
	\end{equation}
	and, in particular,
	\[
	\begin{split}
	& u\leq \overline w\quad \LL^{2n}\mbox{-a.e. in}\ \Omega,\ \forall u\in \cc_{\varphi}\\
	& \underline u\leq w\quad \LL^{2n}\mbox{-a.e. in}\ \Omega,\ \forall w\in \cc_{\psi}.
	\end{split}
	\]
\end{theorem}
Upon observing that $\area_{\varphi+\alpha,\Omega}(u+\alpha)=\area_{\varphi,\Omega}(u)\quad\forall\: u\in \BV(\Omega)$, the following result can be proved exactly as in \cite[Corollary 4.6]{PSTV}.
\begin{corollary}\label{valbordo}
	Let $\Omega\subset\bbR^{2n}$ be a bounded open set with Lipschitz regular boundary and $\varphi,\psi\in L^{\infty }(\partial\Omega)$; let $\overline{u},\underline{u}\in \cc_{\varphi}$ and $\overline{w},\underline{w}\in \cc_{\psi}$ be as in \eqref{uuww}.
	Then, for every $\alpha\in\bbR$, one has
	\begin{equation}\label{traslaz}
	\begin{split}
	&\overline{u}+\alpha, \underline{u}+\alpha\in \cc_{\varphi+\alpha}\\
	&\underline{u}+\alpha\leq u\leq \overline{u}+\alpha\quad \LL^{2n}\mbox{-a.e. in}\ \Omega,\ \forall u\in\cc_{\varphi+\alpha}
	\end{split}
	\end{equation}
	and
	\begin{equation}\label{sup}
	\begin{split}
	& \norm{\overline{u}-\overline{w}}_{L^{\infty}(\Omega)}\leq\norm{\varphi-\psi}_{L^{\infty}(\partial\Omega)}\\
	& \norm{\underline{u}-\underline{w}}_{L^{\infty}(\Omega)}\leq\norm{\varphi-\psi}_{L^{\infty}(\partial\Omega)}.
	\end{split}
	\end{equation}
	In particular, the implications
	\begin{equation}\label{eqrevin}
	\begin{split}
	& \overline u_{|\partial\Omega}=\varphi,\ \overline w_{|\partial\Omega}=\psi\quad \Rightarrow\quad \norm{\overline{u}-\overline{w}}_{L^{\infty}(\Omega)}=\norm{\varphi-\psi}_{L^{\infty}(\partial\Omega)},\\
	& \underline u_{|\partial\Omega}=\varphi,\ \underline w_{|\partial\Omega}=\psi\quad \Rightarrow\quad \norm{\underline{u}-\underline{w}}_{L^{\infty}(\Omega)}=\norm{\varphi-\psi}_{L^{\infty}(\partial\Omega)}.
	\end{split}
	\end{equation}
	hold.
\end{corollary}
	
We recall below some notations introduced in \cite{PSTV}, that will be useful also in the proof of the main theorem of the present paper.
Given a subset $\Omega\subset\bbR^{2n}$, a function $u\colon\Omega\to\R$, a vector $\tau\in\R^{2n}$ and $\xi\in\R$ we set
\[
\begin{split}
& \Omega_{\tau}\coloneqq\{z\in\bbR^{2n}: z+\tau\in \Omega\}\\
& u_{\tau}(z)\coloneqq u(z+\tau),\quad z\in \Omega_\tau\\
& \atx{u}(z)\coloneqq u_{\tau}(z)+2\left\langle \tau^*, z\right\rangle+\xi,\quad z\in \Omega_\tau\,.
\end{split}
\]
It is easily seen that, given $\Omega$ open and $u\in \BV(\Omega)$, then both $u_\tau$ and $\atx u$ belong to $\BV(\Omega_\tau)$. Moreover, if $\Omega$ is bounded with Lipschitz regular boundary one has also
\begin{equation}\label{trennno}
(\atx u)_{|\partial(\Omega_\tau)}= (u_{|\partial\Omega})_{\tau}+2\left\langle \tau^*,\cdot\right\rangle+\xi=\atx{(u_{|\partial\Omega})}\,.
\end{equation}
\begin{oss}{\rm
The family of functions $\atx{u}$ has a precise meaning from the viewpoint of Heisenberg groups geometry. Indeed, it is a matter of computations to observe that the $t$-subgraph $E^t_{\atx u}$ of $\atx{u}$ coincides with the left translation $(-\tau,\xi)\cdot E^t_u$ (according to the group law) of the $t$-subgraph $E^t_u$ of $u$ by the element $(-\tau,\xi)\in \bbH^n$. We address the interested reader to \cite{PSTV,SCV} for further informations.
}\end{oss}

\begin{lemma}\label{lemutile}
Let $\Omega\subset\R^{2n}$ be a bounded open set with Lipschitz regular boundary, $\varphi\in L^1(\partial\Omega)$, $\tau\in\R^{2n}$ and $\xi\in\R$. Then
\[
\area_{\atx{\varphi},\Omega_\tau}(\atx u)=\area_{\varphi,\Omega}(u),\quad\forall u\in \BV(\Omega)\,.
\]
\end{lemma}
\begin{proof}
Using e.g. \cite[Remark 3.18]{AFP}, we get $D u_\tau=\ell_{\tau\#}(Du)$, where $\ell_\tau$ is the translation $z\mapsto z-\tau$ and $\ell_{\tau\#}$ denotes the push-forward of measures via $\ell_\tau$. In particular
\[
\nabla u_\tau=(\nabla u)_\tau=\nabla u\circ\ell_\tau^{-1}, \quad  D^s u_\tau=\ell_{\tau\#}(D^su)\quad \mbox{and}\quad \frac{dD^s u_{\tau}}{d|D^s u_{\tau}|}=\frac{dD^s u}{d|D^s u|}\circ \ell_\tau^{-1}
\]
hence
\[
D\atx u = \big(\nabla u\circ\ell_\tau^{-1}+2\tau^\ast\big)\LL^{2n} + \ell_{\tau\#}(D^su)\,.
\]
Therefore
\[
\begin{split}
& \area_{\atx{\varphi},\Omega_\tau}(\atx u)\\
= & \int_{\Omega_\tau} g((\nabla u\circ\ell_\tau^{-1})+2\tau^\ast+\xva)\ud\LL^{2n} +\int_{\Omega_\tau} g^{\infty}\left( \frac{d D^s u}{d |D^s u|}\circ \ell_{\tau}^{-1}\right)\, \ud|\ell_{\tau\#}(D^su)|\\&+ \int_{\partial\Omega_\tau}g^{\infty}((\atx\varphi-(\atx u)_{|\partial\Omega_\tau})\nu_{\Omega_\tau})\ud\HH^{2n-1}.
\end{split}
\]
We now use \eqref{trennno} and the equality
\[
2\tau^\ast+\xva(z) = 2(\tau+z)^\ast = (\xvett^\ast\circ \ell_\tau^{-1})(z),\quad\forall z\in\R^{2n}
\]
to get, with a change of variable,
\[
\begin{split}
& \area_{\atx{\varphi},\Omega_\tau}(\atx u)\\
= & \int_{\Omega_\tau} |\nabla u+\xva|\circ\ell_\tau^{-1}\ud\LL^{2n} + \int_{\Omega_\tau} g^{\infty}\left( \frac{d D^s u}{d |D^s u|}\circ \ell_{\tau}^{-1}\right)\, \ud|\ell_{\tau\#}(D^su)|+ \int_{\partial\Omega_\tau}g^{\infty}\left(\big(\varphi-u_{|\partial\Omega}\big)_\tau\nu_{\Omega_{\tau}}\right)\ud\HH^{2n-1}\\
=& \int_\Omega |\nabla u + \xva|\ud\LL^{2n} + \int_{\Omega} g^{\infty}\left( \frac{d D^s u}{d |D^s u|}\right)\, \ud|D^su| + \int_{\partial\Omega}g^{\infty}((\varphi-u_{|\partial\Omega})\nu_{\Omega})\ud\HH^{2n-1}\\
= & \area_{\varphi,\Omega}(u)\,.\qedhere
\end{split}
\]
\end{proof}

\begin{corollary}\label{tildemax}
If the same assumptions of Lemma \ref{lemutile} hold and if $\overline{u}$ and $\underline{u}$ are as in Proposition \ref{esistmax}, then $\atx{(\overline{u})},\atx{(\underline{u})}\in \cc_{\atx{\varphi}}$ and
\[
\atx{(\overline{u})} \leq u\leq \atx{(\underline{u})}\quad \LL^{2n}\mbox{-a.e. in }\Omega_{\tau}, \forall u\in \cc_{\atx{\varphi}}\,.
\]
\end{corollary}

The next proposition states that, whenever we fix an affine boundary datum $L$, the functional $\mathcal{G}_{L,\Omega}$ admits as unique minimizer the function $L$ itself. 

\begin{proposition}\label{propmin2}
Let $L\colon \R^{2n}\to\R$ be given by $L(z):=\langle a,z\rangle +b$ with $a\in \R^{2n}$ and $b\in\R$ and assume $g$ satisfies assumptions (A) and (B). Then $L$ is the unique solution of the problem
\begin{equation}\label{minc}
\min\{\mathcal{G}_{L,\Omega}(u): u\in BV(\Omega)\}.
\end{equation}
\end{proposition}
\begin{proof} 
We divide the proof in several steps.

\vspace{1em}
\noindent{\sl Step 1.} We claim there exists $p\colon \R^{2n}\to \R^{2n}$ such that $p(z) \in \partial g(z)$ for any $z\in \R^{2n}$ and with the property that
\begin{equation}\label{key0}
\int_\Omega \langle p(\X^*),\sigma_u\rangle\,\ud|Du|=\int_{\partial \Omega} u_{|\partial\Omega} \langle p(\X^*),\nu_\Omega\rangle\,\ud\mathcal H^{2n-1},
\end{equation}
for any $u\in BV(\Omega)$. 
If $g \in C^2(\R^{2n})$ formula \eqref{key0} with $p=\nabla g$ follows using the Gauss-Green formula and the fact that, since ${\rm div}\, \X^*=0$, also  ${\rm div} \nabla g(\X^*)=0$. We claim that \eqref{key0} holds true again with $p=\nabla g$ if $g\in C^1(\R^{2n})$. Consider the convolutions $g_h\coloneqq\rho_h*g$ where $\rho_h$ is a convolution kernel, i.e.\ $\rho_h\in C_c^\infty(B(0,1/h))$, $\rho_h\geq0$ and $\int_{\mathbb{R}^{2n}}\rho_h=1$. Then $g_h\in C^\infty(\R^{2n})$ and $\nabla g_h\to \nabla g$ uniformly on compact sets. It is now sufficient to pass to the limit in 
\[
\int_\Omega \langle \nabla g_h(\X^*),\sigma_u\rangle\,\ud|Du|=\int_{\partial \Omega} u_{|\partial\Omega} \langle \nabla g_h(\X^*),\nu_\Omega\rangle\,\ud\mathcal H^{2n-1}
\]
using the Dominated Convergence Theorem.
Finally we prove that \eqref{key0} holds true for any convex function $g\colon\mathbb{R}^{2n}\to \mathbb R$ and for a suitable choice of $p$. We are going to use the Yosida approximation; see \cite[Sec.\,IV.1]{S}  (see also \cite[Theorem 2.1]{AA}) for details. Precisely, for any $\lambda>0$ and for any $z\in \R^{2n}$ let 
\[
J_{\lambda}(z)=\min_{y\in\R^{2n}}\left\{\frac{1}{2\lambda}\|y-z\|^2+g(y)\right\},
\]
and
\[
g_{\lambda}(z)=g(J_{\lambda}(z))+\frac{1}{2\lambda}\|z-J_{\lambda}(z)\|^2
\]
Then $g_{\lambda}\in C^{1,1}(\R^{2n})$ and for any $z\in\mathbb{R}^{2n}$ there holds $\nabla f_{\lambda}(z)= A_{\lambda}(z)$ where $A_{\lambda}$ is the Yosida approximation of the maximal monotone operator $A=\partial g$, $A_{\lambda}(z)\coloneqq\lambda^{-1}(z-J_{\lambda}(z))$. Moreover, as $\lambda$ decreases to zero, $g_{\lambda}$ increases to $g$, and for any $z\in\mathbb{R}^{2n}$, $\| A_{\lambda}(z)\|\to \|\partial^0 g(z)\|$ and $A_{\lambda}(z)\to \partial^0 g(z)$, where $\partial^0 g(z)$ denotes the element of minimal norm of the closed convex set $\partial g(z)$. Finally, since $g$ has linear growth we have $\|\partial^0 g(z)\|\le c$ for some $c>0$ and for every $z\in\mathbb{R}^{2n}$.
The thesis now follows by taking  $p\colon \R^{2n}\to \R^{2n}$ defined by $p(z)\coloneqq\partial g^{0}(z)$ and using the Dominated Convergence Theorem to pass to the limit in 
\[
\int_\Omega \langle A_{\lambda}(\X^*),\sigma_u\rangle\,\ud|Du|=\int_{\partial \Omega} u_{|\partial\Omega} \langle A_{\lambda}(\X^*),\nu_\Omega\rangle\,\ud\mathcal H^{2n-1}
\]
as $\lambda\to 0$, obtaining \eqref{key0}.

\vspace{1em}
\noindent{\sl Step 2.} We claim that for any $w,z\in\R^{2n}$ we have
\begin{equation}\label{ginf}
g^{\infty}(w) \ge \langle p(\X^*(z)),w\rangle.
\end{equation}
Indeed, by convexity, for any $t>0$
\begin{equation*}
\frac{g(tw+\X^*(z))}{t}\ge \frac{g(\X^*(z))}{t}+\langle p(\X^*(z)),w\rangle,
\end{equation*}
and the conclusion follows letting $ t \to \infty$ and using Remark \ref{remginf}.

\vspace{1em}
\noindent{\sl Step 3.} We claim that $u=0$ is a solution of the problem
\[
\min\{\mathcal G_{0,\Omega}(u):  u\in BV(\Omega)\}.
\]
Let $u\in BV(\Omega)$. Combining the convexity of $g$ with \eqref{key0} and \eqref{ginf} we obtain
\begin{equation}\label{key4}
\begin{aligned}
\mathcal G_{0,\Omega}(u)&\ge\int_\Omega g(\X^\ast)\ud\LL^{2n}+\int_\Omega \langle p(\X^\ast),\nabla u\rangle\ud\LL^{2n}+\int_\Omega \langle p(\X^\ast),\frac{dD^su}{d|D^su|}\rangle\,\ud|D^su|\\
&\qquad +\int_{\partial\Omega}g^\infty(-u_{|\partial\Omega}\nu_\Omega)\,\ud\mathcal H^{2n-1}\\
&=\int_\Omega g(\X^\ast)\ud\LL^{2n}+\int_\Omega \langle p(\X^\ast),\sigma_u\rangle\,\ud|Du|+\int_{\partial\Omega}g^\infty(-u_{|\partial\Omega}\nu_\Omega)\,\ud\mathcal H^{2n-1}\\
&\ge \int_\Omega g(\X^\ast)\ud\LL^{2n}+\int_{\partial \Omega} u_{|\partial\Omega}\langle p(\X^*),\nu_\Omega\rangle\,\ud\mathcal H^{2n-1}-\int_{\partial \Omega} u_{|\partial\Omega}\langle p(\X^*),\nu_\Omega\rangle\,\ud\mathcal H^{2n-1}\\
&=\int_\Omega g(\X^\ast)\ud\LL^{2n}\\
&=\mathcal G_{0,\Omega}(0)
\end{aligned}
\end{equation}
which ends the proof of the minimality of $u=0$.

\vspace{1em}
\noindent{\sl Step 4.} We claim now that if $\Omega=B_R(0)$ then $u=0$ is the unique solution of the problem
\[
\min\{\mathcal G_{0,\Omega}(u) : u\in BV(\Omega)\}.
\]
Let $u\in BV(\Omega)$ be another minimizer, i.e. $\mathcal G_{0,\Omega}(u)=\mathcal G_{0,\Omega}(0)=m$. By convexity we have
\[
\begin{aligned}
m&\le \mathcal G_{0,\Omega}\left(\frac{u}{2}\right)\\
&=\int_\Omega g\left(\frac{1}{2}\nabla u+\X^\ast\right)\ud\LL^{2n}+\frac{1}{2}\int_\Omega g^\infty\left(\frac{dD^su}{d|D^su|}\right)\,\ud|D^su|+\frac{1}{2}\int_{\partial\Omega}g^\infty(-u_{|\partial\Omega}\nu_\Omega)\,\ud\mathcal H^{2n-1}\\
&\le \frac{1}{2}\int_\Omega g(\nabla u+\X^\ast)\ud\LL^{2n}+\frac{1}{2}\int_\Omega g(\X^\ast)\ud\LL^{2n}+\frac{1}{2}\int_\Omega g^\infty\left(\frac{dD^su}{d|D^su|}\right)\,\ud|D^su|\\
&\qquad +\frac{1}{2}\int_{\partial\Omega}g^\infty(-u_{|\partial\Omega}\nu_\Omega)\,\ud\mathcal H^{2n-1}\\
&=\frac{1}{2}m+\frac{1}{2}m=m.
\end{aligned}
\]
As a consequence we get
\[
g\left(\frac{\nabla u+\X^\ast}{2}+\frac{\X^*}{2}\right)=\frac{g(\nabla u+\X^\ast)+g(\X^*)}{2}, \quad \text{$\mathcal L^{2n}$-a.e.\,on $\Omega$}.
\]
Using assumption (A), we conclude that
\[
\nabla u=\lambda^a \X^*, \quad \text{$\mathcal L^{2n}$-a.e.\,on $\Omega$}
\]
for some measurable function $\lambda^a\colon \Omega \to \R$. Rewriting \eqref{key4} and using \eqref{ginf} we then obtain
\[
\begin{aligned}
m&=\int_\Omega g\left(\nabla u+\X^\ast\right)\ud\LL^{2n}+\int_\Omega g^\infty\left(\frac{dD^su}{d|D^su|}\right)\,\ud|D^su|+\int_{\partial\Omega}g^\infty(-u_{|\partial\Omega}\nu_\Omega)\,\ud\mathcal H^{2n-1}\\
&\ge\int_\Omega g(\nabla u+\X^\ast)\ud\LL^{2n}+\int_\Omega \left\langle p(\X^\ast),\frac{dD^su}{d|D^su|}\right\rangle\,\ud|D^su|+\int_{\partial\Omega}g^\infty(-u_{|\partial\Omega}\nu_\Omega)\,\ud\mathcal H^{2n-1}\\
&\ge \int_\Omega g(\X^\ast)\ud\LL^{2n}+\int_{\partial \Omega} (u_{|\partial\Omega}\langle p(\X^*),\nu_\Omega\rangle+g^\infty(-u_{|\partial\Omega}\nu_\Omega))\,\ud\mathcal H^{2n-1}\\
&\ge\int_\Omega g(\X^\ast)\ud\LL^{2n}\\
&\qquad =m.
\end{aligned}
\]
This means that
\begin{equation}\label{key6}
g^\infty\left(\frac{dD^su}{d|D^su|}\right)=\left\langle p(\X^\ast),\frac{dD^su}{d|D^su|}\right\rangle, \quad \text{$|D^su|$-a.e.\,on $\Omega$}
\end{equation}
and
\begin{equation}\label{key7}
u_{|\partial\Omega}\langle p(\X^*),\nu_\Omega\rangle+g^\infty(-u_{|\partial\Omega}\nu_\Omega)=0, \quad \text{$\mathcal H^{2n-1}$-a.e.\,on $\Omega$}.
\end{equation}
Combining assumption (B) with \eqref{key6}, we immediately deduce that
\[
\frac{dD^su}{d|D^su|}=\lambda^s \frac{\X^*}{|\X^*|}, \quad \text{$|D^su|$-a.e.\,on $\Omega$}
\]
for some measurable function $\lambda^s\colon \Omega \to \R$. From \eqref{key7} we get $u_{|\partial\Omega}=0$. Indeed, at any point of $\partial\Omega$ where $u_{\partial \Omega}>0$, condition \eqref{key7} implies
\[
g^\infty(-\nu_\Omega)=\langle p(\X^*),-\nu_\Omega\rangle
\]
which means, thanks to assumption (B), that $\nu_\Omega$ is parallel to $\X^*$, and this is impossible since $\Omega=B_R(0)$, namely $\X^*\perp \nu_\Omega$ everywhere on $\partial\Omega$. By means of the same argument we can also exclude $u_{|\partial\Omega}<0$. Therefore, we can say that
\[
\sigma_u=\lambda \X^*, \quad \text{$|Du|$-a.e.\,on $\Omega$}
\]
for some measurable function $\lambda\colon \Omega \to \R$. Lemma \eqref{0hom} gives the conclusion.

\vspace{1em}
\noindent{\sl Step 5.} Now we prove that $u=0$ is the unique solution of the problem
\[
\min\{\mathcal G_{0,\Omega}(u) : u\in BV(\Omega)\}
\]
for a general $\Omega$. Indeed, let $u\in BV(\Omega)$ be such that $\mathcal G_{0,\Omega}(u)=\mathcal G_{0,\Omega}(0)$. Let $R>0$ be such that $\Omega \subset\subset B_R(0)$. Let $u_0\colon B_R(0) \to \R$ be given by
\[
u_0(z)\coloneqq\left\{\begin{array}{ll}
u(z) & \textrm{if $z \in \Omega$}\\
0 & \textrm{otherwise.}\end{array}\right.
\]
Then,
\[
\begin{aligned}
\mathcal G_{0,B_R(0)}(u_0)&=\int_\Omega g(\nabla u+\X^*)\ud\LL^{2n}+\int_\Omega g^\infty\left(\frac{dD^su}{d|D^su|}\right)\,\ud|D^su|+\int_{\partial\Omega}g^\infty(-u_{|\partial\Omega}\nu_\Omega)\,\ud\mathcal H^{2n-1}\\
&\qquad +\int_{B_R(0)\setminus \overline \Omega}g(\X^*)\ud\LL^{2n}
\\
&=\mathcal G_{0,\Omega}(u)+\mathcal G_{0,B_R(0)\setminus \overline\Omega}(0)\\
&=\mathcal G_{0,B_R(0)}(0),
\end{aligned}
\]
Where in the last equality we used $\mathcal G_{0,\Omega}(u)=\mathcal G_{0,\Omega}(0)$. Hence,  by step 3 we get $u_0=0$ from which the conclusion.

\vspace{1em}
\noindent{\sl Step 6.} We conclude the proof proving that $u=L$ is the unique solution of the problem
\[
\min\{\G_{L,\Omega}(u) : u\in BV(\Omega)\}.
\]
Let $\Omega_a\coloneqq\Omega-a^*/2$, $u\in BV(\Omega)$ and $u_a \colon \Omega \to \R$ be given by $u_a(z)\coloneqq u(z+a^*/2)-L(z)$. Then $u_a \in BV(\Omega_a)$. Hence we get, using step 2,
\[
\begin{aligned}
\mathcal G_{L,\Omega}(u)&=\int_\Omega g(\nabla u+\X^*)\ud\LL^{2n}+\int_\Omega g^\infty\left(\frac{dD^su}{d|D^su|}\right)\,\ud|D^su|+\int_{\partial\Omega}g^\infty((L-u_{|\partial \Omega})\nu_\Omega)\,\ud\mathcal H^{2n-1}\\
&=\int_{\Omega_a}g(\nabla u_a+\X^*)\ud\LL^{2n}+\int_{\Omega_a} g^\infty\left(\frac{dD^su_a}{d|D^su_a|}\right)\,\ud|D^su_a|+\int_{\partial\Omega_a}g^\infty(-(u_a)_{|\partial \Omega}\nu_\Omega)\,\ud\mathcal H^{2n-1}\\
& =\mathcal G_{0,\Omega_a}(u_a)\ge \mathcal G_{0,\Omega_a}(0)\\
&=\int_{\Omega_a}g(\X^*)\ud\LL^{2n}=\int_\Omega g(a+\X^*)\ud\LL^{2n}\\
&=\G_{L,\Omega}(L)
\end{aligned}
\]
which says that $u=L$ is a minimizer. Uniqueness easily follows by the fact that the equality $\mathcal G_{L,\Omega}(u)=\mathcal G_{L,\Omega}(0)$  implies, using the previous estimate, $\mathcal G_{0,\Omega_a}(u_a)=\mathcal G_{0,\Omega_a}(0)$ which in turn yields $u_a=0$ from step 4. In order to conclude the proof it is sufficient to observe that $u_a=0$ means $u=L$.
\end{proof}

\begin{corollary}\label{cpaffine}
Let $\Omega\subset\bbR^{2n}$ be a bounded open set with Lipschitz boundary, $\varphi\in L^1(\partial\Omega)$ and $L\colon\bbR^{2n}\to\bbR$ be an affine function, i.e., $L(z)=\langle a,z\rangle+b$ for some $a\in\R^{2n},b\in\bbR$.
\begin{enumerate}
\item Assume that $\varphi\leq L$\, $\HH^{2n-1}$-a.e. on $\partial\Omega$. Then, for any $u\in \cc_{\varphi}$, we have $u\le L$ \,$\LL^{2n}$-a.e.\ in $\Omega$.
\item Assume that that $\varphi\geq L$ \,$\HH^{2n-1}$-a.e. on $\partial\Omega$. Then, for any $u\in \cc_{\varphi}$, we have $u\ge L$ \,$\LL^{2n}$-a.e.\ in $\Omega$.
\end{enumerate}
\end{corollary}
\begin{proof}
Both claims follow immediately from Theorem \ref{confro} when we
observe that the set $\cc_L$ consists of just one element
that is $L$ itself, so that, following
the notations of Proposition \ref{esistmax}, $L=\overline L=\underline L$.
\end{proof}

\section{The Bounded Slope Condition}\label{ultSec}
We recall the well-known definition of a boundary datum satisfying
the Bounded Slope Condition (see \cite{HartStamp}). We also refer to \cite{G} for some classical results.
\begin{definition}\label{BSC}
{ We say that a function $\varphi\colon\partial\Omega\to\bbR$ satisfies the bounded slope condition with constant $Q>0$ ($Q$-B.S.C.
for short or simply B.S.C.\ when the constant $Q$ does not play any role) if for every $z_0\in\partial\Omega$, there exist two affine functions $w^+_{z_0}$ and $w^-_{z_0}$ such that
\begin{align}
\label{compar}
&w^-_{z_0}(z)\leq \varphi(z)\leq w^+_{z_0}(z) \quad \forall z\in\partial\Omega,\\
\label{eqpunto}
&w^-_{z_0}(z_0)=\varphi(z_0)=w^+_{z_0}(z_0)\\
\label{stimali} &\lip(w^-_{z_0})\leq Q\quad \mbox{and}\quad
\lip(w^+_{z_0})\leq Q,
\end{align}
where $\lip(w)$ denotes the Lipschitz constant of $w$.}

Moreover, we denote by $f_1$ and $f_2$ the functions defined, respectively, by
$f_1(z)\coloneqq\sup_{z_0\in\partial\Omega} w^-_{z_0}(z)$ and
$f_2(z)\coloneqq\inf_{z_0\in\partial\Omega} w^+_{z_0}(z)$. We underline that $f_1$
is convex, $f_2$ is concave and they are both
Lipschitz continuous with Lipschitz constant not greater than $Q$.
\end{definition}

The following result can be proved exactly as in \cite[Lemma 6.2]{PSTV}.
\begin{lemma}\label{datobordo}
Let $\Omega\subset\bbR^{2n}$ be an open bounded set with Lipschitz regular boundary; assume that $\varphi\in L^1(\partial\Omega)$ satisfies the $Q$-B.S.C. Then, if $u\in BV(\Omega)$ is a minimizer of $\mathcal{G}_{\varphi,\Omega}$, the following facts hold.
\begin{enumerate}
\item  $u_{|\partial\Omega}=\varphi$;
\item $f_1\le u\le f_2$ $\LL^{2n}$-a.e. in $\Omega$;
\item $u$ is also a minimizer of  $\area_{\Omega}$  in $\BV(\Omega)$.
\end{enumerate}
\end{lemma}


The following fact is inspired by \cite[Remark 6.4]{PSTV}.
\begin{remark}\label{sottodomini}{\rm
If $\Omega'\subset\Omega$ are open bounded domains with Lipschitz regular boundary and $u\in \BV(\Omega)$.

Write $\Gamma\coloneqq\partial\Omega'\cap\Omega$ and $\partial\Omega=\Delta_1\cup\Delta_2$, where
\[
\Delta_1\coloneqq \partial\Omega\cap\partial\Omega'\quad\text{and}\quad \Delta_2\coloneqq \partial\Omega\setminus\partial\Omega'\,.
\]
Notice that $\partial\Omega'=\Gamma\cup\Delta_1$. We also denote by $u_i,u_o\colon\Gamma\to\R$ the ``inner'' and ``outer'' (with respect to $\Omega'$) traces of $u$ on $\Gamma$, i.e.,
\[
u_i\coloneqq(u_{|\partial\Omega'})\res\Gamma\quad\text{and}\quad u_o\coloneqq(u_{|\partial(\Omega\setminus\overline{\Omega'})})\res\Gamma\,.
\]
We use the notation $\area_{u,\Omega'}$ to denote the functional $\area_{u_o,\Omega'}$. Let us prove that, if $u$ is a minimizer of $\area_{\varphi,\Omega}$ with $\varphi=u_{|\partial\Omega}$, then $u$ is also a minimizer of $\area_{u,\Omega'}$.
Assume by contradiction that $u$ is not a minimizer of $\area_{u,\Omega'}$; then, there exists $v\in \BV(\Omega')$ such that
\begin{equation}\label{ristorantebrasiliano}
\begin{split}
0 & < \area_{u,\Omega'}(u)-\area_{u,\Omega'}(v)\\
&= \area_{\Omega'}(u)-\area_{\Omega'}(v)+\int_\Gamma g^\infty((u_o-u_i)\nu_{\Omega'})\ud\HH^{2n-1}\\
& \hphantom{=\area_{\Omega'}(u)-\area_{\Omega'}(v)}-\int_\Gamma g^\infty((u_o-v_{|\partial\Omega'})\nu_{\Omega'})\ud\HH^{2n-1}-\int_{\Delta_1} g^\infty((\varphi-v_{|\partial\Omega'})\nu_\Omega)\ud\HH^{2n-1}\\
\end{split}
\end{equation}
where we used inequality \eqref{sublin}. We would reach a  contradiction if we show that the function $w\in\BV(\Omega)$ defined by
\[
w\coloneqq v\text{ on }\Omega',\quad w\coloneqq u\text{ on }\Omega\setminus\Omega'
\]
satisfies $\area_{\varphi,\Omega}(u)- \area_{\varphi,\Omega}(w)>0$.

Let us compute
\[
\begin{split}
\area_{\varphi,\Omega}(u) & = \area_{\Omega}(u) =  \area_{\Omega'}(u) + \area_{\Omega\setminus\overline{\Omega'}}(u) + \int_{\Gamma} g^{\infty}\left(\frac{dD^s u}{d|D^s u|}\right) \ud|D^s u|\\
& =   \area_{\Omega'}(u) + \area_{\Omega\setminus\overline{\Omega'}}(u) + \int_\Gamma g^{\infty}((u_o-u_i)\nu_{\Gamma})\ud\HH^{2n-1}
\end{split}
\]
and
\[
\begin{split}
\area_{\varphi,\Omega}(w) &=   \area_{\Omega'}(v) + \area_{\Omega\setminus\overline{\Omega'}}(u) + \int_{\Gamma} g^{\infty}\left(\frac{dD^s w}{d|D^s w|}\right) \ud|D^s w|+ \int_{\partial\Omega} g^{\infty}((\varphi-w_{|\partial\Omega})\nu_{\Omega})\ud\HH^{2n-1}\\
&=   \area_{\Omega'}(v) + \area_{\Omega\setminus\overline{\Omega'}}(u) + \int_\Gamma g^{\infty}(u_o-v_{|\partial\Omega'})\nu_{\Omega'})\ud\HH^{2n-1} + \int_{\Delta_1} g^{\infty}(\varphi-v_{|\partial\Omega})\nu_{\Omega})\ud\HH^{2n-1}\,.
\end{split}
\]
Therefore
\[
\begin{split}
 \area_{\varphi,\Omega}(u) &-\area_{\varphi,\Omega}(w)\\
= \,& \area_{\Omega'}(u)-  \area_{\Omega'}(v) + \int_\Gamma \Big(g^{\infty}((u_o-u_i)\nu_{\Omega'}) - g^{\infty}(u_o-v_{|\partial\Omega'})\nu_{\Omega'})\Big)\ud\HH^{2n-1} \\
&\qquad - \int_{\Delta_1} g^{\infty}((\varphi-v_{|\partial\Omega'})\nu_{\Omega})\ud\HH^{2n-1} >0,
\end{split}
\]
where we used \eqref{ristorantebrasiliano} and $u_{|\partial\Omega'}=u_i$.
}\end{remark}

We are now in position to prove our main result, whose proof is actually very similar to the one given in \cite{PSTV}.

\begin{theorem}\label{mainteo-ex6.7}
Let $\Omega\subset\bbR^{2n}$ be open, bounded and with Lipschitz regular boundary, let $\varphi\colon\partial\Omega\to\R$ satisfy the $Q$-B.S.C. for some $Q>0$ and let $g\colon\mathbb{R}^{2n}\to\mathbb{R}$ be a convex function with linear growth satisfying conditions $(A)$ and $(B)$.  Then, the minimization problem
\begin{equation}\label{problema}
\min\left\{\mathcal{G}_{\Omega}:
u\in \BV(\Omega),\ u_{|\partial\Omega}=\varphi\right\}
\end{equation}
admits a unique solution $\hat{u}$. Moreover, $\hat u$ is Lipschitz continuous and $\lip(\hat{u})\leq \overline Q=\overline Q(Q,\Omega)$.
\end{theorem}

\begin{proof} We divide the proof into several steps.

{\em Step 1.} We denote by $\overline u$ the (pointwise a.e.) maximum of the minimizers of $\area_{\varphi,\Omega}$ in $\BV$ (see Proposition \ref{esistmax}). Lemma \ref{datobordo} implies that $f_1\le \overline u\le f_2$ $\LL^{2n}$-a.e.\ in $\Omega$ and $\overline u=\varphi=f_1=f_2$ on $\partial\Omega$, where $f_1$ and $f_2$ are defined as in Definition \ref{BSC}; in particular, $\overline u$ is also a minimizer for \eqref{problema}.

Let $\tau\in\R^{2n}$ be such that $\Omega\cap\Omega_\tau\neq\emptyset$; following the notations introduced before Lemma \ref{lemutile}, we consider the function $\overline u_{\tau,0}^*$, which we denote by $\overline u_\tau^*$ to simplify the notation. Consider the set $\Omega\cap\Omega_\tau$. By Remark \ref{sottodomini}, $\overline u$ is a minimizer of $\area_{\overline u,\Omega\cap\Omega_\tau}$ and, by Corollary \ref{tildemax} and Remark \ref{sottodomini}, $\ut$ is a minimizer of $\area_{\ut,\Omega\cap\Omega_\tau}$. Let $z\in\partial(\Omega\cap\Omega_\tau)$, then either $z\in\partial\Omega$ or $z\in\partial\Omega_\tau$.

\noindent If $z\in \partial \Omega$, then $z+\tau\in\overline\Omega$ and the inequality $(36)$ in \cite[Lemma $6.3$ ]{PSTV} implies that
\begin{equation}\label{disbordo}
\overline u(z)-Q|\tau|\le \overline u(z+\tau)\le \overline u(z)+Q|\tau|\,.
\end{equation}
Otherwise, $z\in\partial\Omega_\tau$ and $z=(z+\tau)-\tau\in\overline\Omega$, and Lemma \ref{datobordo} implies again \eqref{disbordo}.

So we have proved that \eqref{disbordo} holds for any $z\in\partial (\Omega\cap\Omega_\tau)$, hence
\[
\overline u(z)-Q|\tau|+2\langle\tau^*,z\rangle\le \overline u(z+\tau)+2\langle\tau^*,z\rangle\le \overline u(z)+Q|\tau|+2\langle\tau^*,z\rangle\,.
\]
Setting $M\coloneqq Q+2\sup_{z\in\Omega}|z|$, one has
\[
\overline u(z)-M|\tau|\le  \ut(z)\le \overline u(z)+M|\tau|\quad \text{for any }z\in\partial(\Omega\cap\Omega_\tau)
\]
and, by Corollary \ref{valbordo},
\[
\overline u(z)-M|\tau|\le  \ut(z)\le \overline u(z)+M|\tau|\quad \text{for $\LL^{2n}$-a.e. }z\in\Omega\cap\Omega_\tau\,.
\]
This is equivalent to
\[
\overline u(z)-M|\tau|-2\langle \tau^*,z\rangle\le  \overline u(z+\tau)\le \overline u(z)+M|\tau|-2\langle \tau^*,z\rangle\quad \text{for $\LL^{2n}$-a.e. }z\in\Omega\cap\Omega_\tau
\]
and, setting $K\coloneqq M+2\sup_{z\in\Omega}|z|$,
\[
\overline u(z)-K|\tau|\le  \overline u(z+\tau)\le \overline u(z)+K|\tau|\quad \text{for $\LL^{2n}$-a.e. }z\in\Omega\cap\Omega_\tau.
\]

{\em Step 2.} We claim that the inequality $|\overline u(z)-\overline u(\bar z)|\le K|z-\bar z|$ holds for any Lebesgue points $z,\bar z$ of $\overline u$. We define $\tau\coloneqq\bar z-z$; then $\Omega\cap\Omega_\tau\neq\emptyset$ and, arguing as in Step 1, we obtain
\[
|\overline u(z'+\tau)-\overline u(z')|\le K|\tau| \quad\text{for $\LL^{2n}$-a.e. }z'\in\Omega\cap\Omega_\tau.
\]
Let $\rho>0$ be such that $B(z,\rho)\subset\Omega\cap\Omega_\tau$ and $B(\bar z,\rho)\subset\Omega\cap\Omega_\tau$; then
\[
\begin{split}
|\overline u(z)-\overline u(\bar z)|=&\left|\lim_{\rho\to 0}\left(\fint_{B(z,\rho)}\overline u(z') dz'-\fint_{B(\bar z,\rho)}\overline u(z') dz'\right)\right|\\
\le&\lim_{\rho\to 0}\fint_{B(z,\rho)}\left| \overline u(z') - \overline u(z'+\tau) \right|dz'\le K|z-\bar z|.
\end{split}
\]

{\em Step 3.} We have proved that $\overline u$, the  maximum of the minimizer of $\area_{\varphi,\Omega}$, has a representative that is Lipschitz continuous on $\Omega$, with Lipschitz constant not greater than $K= Q+4\sup_{z\in\Omega}|z|$. The same argument leads to prove that $\underline u$, the minimum of the minimizers of $\area_{\varphi,\Omega}$, has a representative that is Lipschitz continuous on $\Omega$, with Lipschitz constant not greater than $K$. The uniqueness criterion in Proposition \ref{uniq-special} (with $p=1$) implies that $\overline u=\underline u$ $\LL^{2n}$-a.e.\ on $\Omega$. If $u$ is another minimizer of $\area_{\varphi,\Omega}$, we have by Proposition \ref{esistmax} that $\underline u\le u\le\overline u$  $\LL^{2n}$-a.e.\ on $\Omega$. This concludes the proof.
\end{proof}

\section{The superlinear growth case}\label{lastSec}

In this section we consider the functional defined in (\ref{sobolev}) by
\begin{equation}\label{gsobolev}
\G_{\Omega}(u)\coloneqq\int_\Omega g(\nabla u+\X^*)\ud\LL^{2n},\qquad u\in\varphi+ W_0^{1,1}(\Omega)
\end{equation}
where $\varphi$ satisfies, as in the previous sections, the Bounded Slope Condition of order $Q$ and $g$ has superlinear growth.


Our aim is to show that, for the functional $\mathscr{G}_{\Omega}$ defined in \eqref{gsobolev}, we can get both regularity and uniqueness results using again the Bounded Slope Condition and arguing with the same approach that we used for the BV case.

\begin{theorem}\label{mainsobolev}
Let $g\colon\R^{2n}\rightarrow\R$ be a convex function satisfying condition (A) and let $\varphi\colon\Omega\rightarrow\R$ satisfy the Bounded Slope Condition of order Q on the boundary of $\Omega$. Assume also that $g$ has superlinear growth, i.e., $g(\xi)\ge \psi(|\xi|)$ for a suitable $\psi\colon[0,+\infty)\rightarrow\R$ such that
\[
\lim_{t\to+\infty}\frac{\psi(t)}{t}=+\infty.
\]
Then the functional
\begin{equation}
  \mathscr{G}_\Omega(u)=\int_{\Omega}g(\nabla u+\X^*)\ud\LL^{2n},\qquad u\in\varphi+W_0^{1,1}(\Omega)
\end{equation}
has a unique Lipschitz minimizer, i.e.: there exists $u\in\varphi+W_0^{1,\infty}(\Omega)$ such that $\mathscr{G}_\Omega
(u)\le\mathscr{G}_\Omega(v)$ for every $v\in\varphi+W_0^{1,1}(\Omega)$.
\end{theorem}
\begin{proof} The superlinearity of $g$ and the lower semicontinuity of $\mathscr{G}_\Omega$ imply the existence of $u_0\in\varphi+W_0^{1,1}(\Omega)$ such that $\mathscr{G}_\Omega
(u_0)\le\mathscr{G}_\Omega(u)$ for every $u\in\varphi+W_0^{1,1}(\Omega)$.

In the same spirit of the previous sections, we denote by $\mathscr{M}_{\varphi}=\{v\in \varphi+W_0^{1,1}(\Omega): \mathscr{G}_\Omega(v)\le \mathscr{G}_\Omega(u), \forall u\in \varphi+W_0^{1,1}(\Omega)\}$.  Thanks to the superlinearity of $g$ we can argue as in the proof of Proposition \ref{esistmax} to state that there exist two functions $\overline{u},\underline{u}\in \mathscr{M}_{\varphi}$ such that for every $u\in\mathscr{M}_{\varphi}$
\[
	\underline{u}(x)\le u(x)\le\overline{u}(x)\quad\text{for a.e.\ $x\in$ }\Omega.
\]
We remark that the results contained in sections \ref{sec:linear} and \ref{ultSec} can be restated replacing the space $BV(\Omega)$ with $\varphi+W_0^{1,1}(\Omega)$. All the proofs in fact can be repeated and simplified dropping both the terms where $D^s$ appears and those that take into account the jumps at the boundary. Hence we can conclude that $\overline{u}\in \varphi+W_0^{1,\infty}(\Omega)$, where $K=Q+2\max_\Omega|z|$. Proposition \ref{uniq-special} then leads to uniqueness of minimizers.
\end{proof}

\begin{remark}
  (1) We underline that regularity results are usually obtained under ellipticity and growth conditions on the Lagrangian. In the present paper, the Bounded Slope Condition allows us to prove Lipschitz regularity up to the boundary with the same flavor of results, sometimes called ``semiclassical'' for the case of functionals depending only on the gradient. First of all we recall a result that goes back to some special cases of Hilbert and Haar and can be found in \cite{G}. The theorem proved in \cite{G} states the existence of a minimizer when the class of competitor functions coincides with the class of Lipschitz functions. The fact that autonomous scalar functionals do not exhibit the Lavrentiev phenomenon (see \cite{BuBe} for some special cases and \cite{BMT1, BMT2} for more general results) implies also that the minimum in the space of Lipschitz functions is also a minimum in $\varphi+W_0^{1,1}(\Omega)$.
    In the present paper we used an approach inspired by \cite{CellinaBSC} where, again in the case of functionals depending only on the gradient and assuming the existence of a minimizer, the Lipschitz regularity follows thanks to the use of Comparison Principles and of the Bounded Slope Condition. As far as we know, the only papers where this result has been extended to functional depending also on the $x$-variable are \cite{PSTV} where the authors considered the area functional in the Heisenberg group and \cite{FiaschiTreu} where the Lagrangian has the form $f(\xi)+g(x,u)$.

  \noindent(2) A consequence of our result in the case of Sobolev spaces is that the assumptions of Theorem \ref{mainsobolev} guarantee the non occurrence of the Lavrentiev phenomenon. This result is classically obtained under suitable assumptions that control from above the growth of the functional. A recent result on a class of functional that includes those considered in this paper is \cite{MT1}.
\end{remark}

\end{document}